\documentclass[reqno]{amsart}

\usepackage{amsmath,amssymb,amsthm,amsfonts}
\usepackage[english]{babel}

\usepackage{cancel}
\usepackage{color}
\usepackage{inputenc}
\usepackage{fontenc}
\usepackage{tikz}

\newcommand{\cD}{{\mathcal D}}

\newcommand{\cL}{{\mathcal L}}

\newcommand{\cQ}{{\mathcal Q}}

\newcommand{\beq}{\begin{equation}}
\newcommand{\bey}{\begin{eqnarray}}
\newcommand{\beyy}{\begin{eqnarray*}}
\newcommand{\eeq}{\end{equation}}
\newcommand{\eey}{\end{eqnarray}}
\newcommand{\eeyy}{\end{eqnarray*}}

\newcommand{\myspace}{\qquad\qquad\qquad}

\newtheorem{theorem}{Theorem}[section]
\newtheorem{lemma}[theorem]{Lemma}
\newtheorem{proposition}[theorem]{Proposition}
\newtheorem{remark}[theorem]{Remark}
\newtheorem{remarks}[theorem]{Remarks}

\newtheorem{assumptions}[theorem]{Assumptions}

\newtheorem{problem}[theorem]{Problem}
\numberwithin{equation}{section}
\date{}

\begin{document}

\title[]{Uniqueness for Riccati equations with unbounded operator coefficients}

\author{Paolo Acquistapace}
\address{Universit\`a di Pisa, Dipartimento di Matematica,
Largo Bruno~Pontecorvo 5, 56127 Pisa, ITALY 
}
\email{paolo.acquistapace(at)unipi.it}

\author{Francesca Bucci}
\address{Francesca Bucci, Universit\`a degli Studi di Firenze,
Dipartimento di Matematica e Informatica,
Via S.~Marta 3, 50139 Firenze, ITALY
}
\email{francesca.bucci(at)unifi.it}

%\subjclass[2020]{\dots}

%\keywords{\dots}

\begin{abstract}
In this article we address the issue of uniqueness for differential and algebraic operator Riccati equations,
under a distinctive set of assumptions on their unbounded coefficients.
The class of boundary control systems characterized by these assumptions 
encompasses diverse significant physical interactions,
all modeled by systems of coupled hyperbolic/parabolic partial differential equations.
The proofs of uniqueness provided tackle and overcome the obstacles raised by the peculiar regularity properties of the composite dynamics.
These results supplement the theories of the finite and infinite time horizon linear-quadratic problem devised by the authors jointly with Lasiecka, as the unique solution to the Riccati equation enters the closed loop form of the optimal control.
\end{abstract}

\maketitle

% INTRO

\section{Introduction}
Well-posedness of Riccati equations is a fundamental question within control theory of Partial Differential Equations (PDE).
While the issues of existence and uniqueness for the corresponding solutions are both natural to be
addressed and significant in themselves, when seen in the context of linear-quadratic optimal control uniqueness proves particularly relevant, not exclusively from a theoretical perspective.
This is because it brings about {\em in a univocal manner} the (optimal cost, or Riccati) operator which occurs in the feedback representation of the optimal control, thereby allowing its synthesis.

In the present work focus is on the differential and algebraic Riccati equations
arising from optimal control problems with quadratic functionals for the class of infinite dimensional
abstract control systems dealt with in our earlier works \cite{abl_2005} and \cite{abl_2013},
joint with Lasiecka.
The basic characteristics of these linear systems -- which read as $y'=Ay+Bu$, according to a standard notation  -- are the following: the free dynamics operator $A$ is the infinitesimal generator of a $C_0$-semigroup $\{e^{At}\}_{t\ge 0}$ on the state space $Y$, while the control operator $B$ is {\em unbounded},
meaning that $B$ maps continuously the control space $U$ into a larger functional space than $Y$,
that is the extrapolation space $[\cD(A^*)]'$; see the Assumptions~\ref{a:ipo_0}.
It is well known that the latter is an intrinsic feature of differential systems describing evolutionary PDE with boundary (and also point) control, already recognized by the end of the sixties in the pioneering work of Fattorini \cite{fattorini_1968}.

We require in addition and more specifically that several assumptions on the operators $A, B$ are fulfilled, 
recorded in Section~\ref{s:framework} as Assumptions~\ref{a:ipo_1}.
These are regularity properties that pertain to the operator $B^*e^{A^*t}$, with
respective PDE counterparts. 
It is worth emphasizing here that not only the aforesaid control-theoretic properties do not prescribe analyticity of the semigroup $e^{At}$, in accordance with the fact that the class of systems under consideration -- introduced by these authors with Lasiecka in \cite{abl_2005} --
is inspired by and tailored on systems of coupled hyperbolic-parabolic PDE, subjected to
boundary/interface control.
They are also weaker (and somewhat trickier) than the full {\em singular estimates} for $e^{At}B$ which are known to be equally effective for the study of the quadratic optimal control problem on both a finite and infinite time horizon, as proved in \cite{las-cbms}, \cite{las-trento}, \cite{las-trig-se-1,las-trig-se-2}.
In this respect, we remark that it was first discovered in \cite{bucci-las-thermo_2004} that a very same 
% Kirchhoff thermoelastic plate model 
thermoelastic system %, subjected to boundary (thermal) control, 
may or may not yield a singular estimate for the corresponding operator $e^{At}B$, depending on which 
combination of boundary conditions/actions (e.g., clamped/Neumann {\em vs} clamped/Dirichlet)
is taken into consideration.

Another distinguishing feature of the coefficients of the Riccati equations under study,
which when the unknown $P$ is time-independent read as 
\begin{equation}\label{e:ARE}
(Px,Az)_Y+(Ax,Pz)_Y-(B^*Px,B^*Pz)_U+(Rx,Rz)_Z=0\,,  \quad x, z\in \cD(A)\,,
\end{equation}
is that $R$ -- that is the {\em observation} operator in the optimization problem -- does not need to be smoothing; see \eqref{e:keyasR}, that is iiib) of the Assumptions~\ref{a:ipo_1}.

Working in the framework described above, a theory for both the finite and infinite time horizon
LQ-problem has been devised in \cite{abl_2005} and \cite{abl_2013}, the latter under the Assumptions~\ref{a:ipo_2} (replacing Assumptions~\ref{a:ipo_1}).
The strenght of these theories is confirmed by the trace regularity results that have been established for the solutions to significant PDE systems comprising hyperbolic and parabolic components, over the years. 
Indeed, the novel class introduced in \cite{abl_2005} has proven successful in describing a diverse range of physical interactions such as mechanical-thermal, acoustic-structure, fluid-elasticity ones.
And above all, successful in order to attain solvability of the associated optimization problems;
see \cite{abl-thermo_2005}, \cite{bucci-applicationes}, \cite{bucci-las-fsi_2010,bucci-las-fsi_2011},
\cite{abl_2013}, and the recent \cite{bucci_2020}.

We recall that since in \cite{abl_2005} and \cite{abl_2013} we followed a variational approach, 
by using the optimality conditions a bounded operator -- to wit, $P(t)$ or $P$, in accordance with 
either a finite or infinite time interval -- is constructed in terms of the optimal state and only subsequently shown to satisfy the corresponding Riccati equations.
For this reason the works \cite{abl_2005} and \cite{abl_2013} provide existence for Riccati equations, but not uniqueness.

By contrast and as it is well-known, in a (sometimes called) `direct' approach, 
the well-posedness of the nonlinear Riccati equation is studied in a first step, %posssibly 
independently from the minimization problem.
For the Cauchy problems associated with the differential Riccati equations, 
uniqueness is established along with existence, 
as joint outcomes of a method of proof based on fixed point theorems. 
% (and next on {\em a priori} estimates).
%
Then, in order to achieve the actual feedback representation of the optimal control, that is the ultimate goal from a mathematical as well as a practical perspective, further steps are necessary.
(This is not the case here.) 

Thus, our aim with the present work is to provide {\em uniqueness} results for both the differential and algebraic Riccati equations, thereby completing the complex of findings of \cite{abl_2005} and \cite{abl_2013}, respectively.
These results -- our main results -- are stated as Theorems~\ref{t:dre-uniqueness} and \ref{t:are-uniqueness}.
Uniqueness holds in appropriate (respective) classes of linear bounded operators, which are consistent
with the distinctive property of the {\em gain operator} emerged in \cite{abl_2005,abl_2013};
see the statements S4. and A4. in Theorem~\ref{t:theory-tfinite} and Theorem~\ref{t:theory-tinfinite}, respectively.

We feel it is important to emphasize that although often following approaches and arguments previously used in the past literature, our proofs of uniqueness need to face novel challenges.
This is due to the specific assumptions on the operator coefficients of the Riccati equations,
notably weaker than the ones which characterize the parabolic class (and its generalization),
somewhat a little bit stronger than mere admissibility which is typical of the hyperbolic class,
yet not requiring smoothness of the observations.

We find it useful to devote a separate section to several comments and points concerning our analysis.
%
% MATH PROOFS
%
\subsection{An insight into the mathematical proofs} \label{ss:insight}
We provide two proofs of Theorem~\ref{t:dre-uniqueness}, that pertains to uniqueness for the 
differential Riccati equations \eqref{e:DRE} (DRE, in short).
This result is relevant for the optimal control problem on a finite time horizon (i.e. Problem~\ref{p:problem-0} with $T<+\infty$), under the Assumptions~\ref{a:ipo_1}.

The first proof, given in Section~\ref{s:uniqueness_1}, follows the method employed
by Lasiecka and Triggiani in \cite[Vol.~I, Theorem~1.5.3.3; Vol.~II, Theorems~8.3.7.1]{las-trig-redbooks}, up to a certain point.
The basic rationale is standard: one proceeds by contradiction, assuming there exists another solution 
$P_1(t)$ to the DRE, besides the optimal cost operator $P(t)$.
(In our case, that $P(t)$ solves the DRE has been proved in \cite{abl_2005}; see the statement S6. of 
Theorem~\ref{t:theory-tfinite}.)
On the basis of the integral form of the DRE -- in the present case two forms, derived in Lemma~\ref{l:DtoI} --, one finds that the difference $Q(t)=P_1(t)-P(t)$ solves a suitable integral equation.
It is in the estimates performed afterwards, that the paths diverge, with iiic) of the
Assumptions~\ref{a:ipo_1} playing a major role here, together with the class of operators 
$P(t)$ and $P_1(t)$ belong to.
The proofs carried out in the aforementioned results instead take advantage of either the enhanced regularity of the analytic semigroup $e^{At}$ that describes the free dynamics, or the additional regularity assumed on
the operators $Re^{At}$ and $Re^{At}B$. 

It is unlikely that the method of proof described above could be adjusted in order to establish uniqueness
for the algebraic Riccati equation, relevant for Problem~\ref{p:problem-0} with $T=+\infty$.
This owing to the argument employed when $T< +\infty$: $Q(t)$ is shown to be zero on some subinterval
of $[0,T]$, with the soughtafter goal attained in a finite number of steps.
Other methods of proof are certainly worth to be explored.
One might attempt to proceed along the lines of the proof of \cite[Vol.~I, Theorem~2.4.5]{las-trig-redbooks}, despite the absence of analyticity of the semigroup $e^{At}$, as well as of the optimal state semigroup 
$\Phi(t)$.
If this were the case, a preliminary analysis which appears unavoidable would pertain to issues
connected to a given solution $P_1$ to the ARE (a priori, distinct from the optimal cost operator $P$),
such as the possible generation of a $C_0$-semigroup on $Y$ by the operator $A-BB^*P_1$, 
in turn to be suitably defined. 
It is unclear in this respect whether a {\em full} description of the domain of the generator itself $A_P$
of $\Phi(t)$, would be required.
Improving the statement A5. of Theorem~\ref{t:theory-tinfinite} is an independent interesting subtlety;
in any case, we leave this question open.

\smallskip 
Thus, in order to prove Theorem~\ref{t:are-uniqueness} we choose to return to 
the dynamic programming approach to the LQ-problem, and borrow from it a key element in attaining
that the optimal control admits a (pointwise in time) feedback representation.
This element is fulfilled by the so called {\em fundamental identity}.
In a direct approach, the fundamental identity builds a bridge between the nonlinear Riccati equation -- whose well-posedness is studied in a first step, independently from the minimization problem, as recalled above -- and the actual closed loop form of the optimal control.
The latter goal (i.e. the feedback representation of the optimal control) was already attained in \cite{abl_2005} and \cite{abl_2013}; see the statements 
S4. of in Theorem~\ref{t:theory-tfinite} and A4. of Theorem~\ref{t:theory-tinfinite}, respectively.
And yet, the identities we establish in Lemma~\ref{l:id-fondamentale} and Lemma~\ref{l:id_fond_infty}
constitute a major (and technically nontrivial) step in our analysis, allowing to achieve uniqueness for both differential and algebraic Riccati equations, respectively.
Theorem~\ref{t:dre-uniqueness} and Theorem~\ref{t:are-uniqueness} are thus established,
via methods of proof which are akin to each other.
%
% Subsection: HISTORICAL SYNOPSIS
%
\subsection{Riccati equations: a historical synopsis, and a few open questions}
Historically, the appearance of {\em matrix} Riccati equations -- named after Jacopo Riccati (XVII century) -- as a research subject is recognized to date back to the sixties, with the independent contributions of Kalman \cite{kalman_1960} in the USA and of Letov \cite{letov_1961} in the former Soviet Union. 
Their study has already lasted for more than half a century, owing to its connections to a wide variety of topics such as stability and stabilization problems, the linear-quadratic (LQ) optimal control problem, differential games, just to name a few.
We refer the reader to the introductory monograph \cite{zabczyk-book}, along with its references. 

We begin by briefly mentioning a broader problem,
and then we focus on well-posedness of Riccati equations with unbounded operator coefficients.
We make an attempt to retrace the main stages of its evolution, which spans over decades, thereby
arriving at the problem at hand.

\subsubsection{The non-standard LQR-problem, matrix/operator inequalities}
A related significant problem which sees its development around the same years is the one referred
to in the literature as the {\em non-standard} (linear) quadratic regulator problem.
Associated with the Cauchy problems for the linear control system $y'=Ay+Bu$ is a quadratic
functional, which may not be coercive; 
it is indeed the integral of a general quadratic form 
\begin{equation*}
F(y,u)= \langle Qy,y \rangle + 2 \text{Re} \langle S u,y \rangle + \langle Ru,u \rangle\,,
\end{equation*}
where for instance $Q$ is allowed to be negative definite.
The goal, in short, is then to characterize -- given an arbitrary initial state $y_0$ -- the condition
\begin{equation*}
\inf_{u\in L^2(0,T;U)}\int_0^T F(y(t),u(t))\,dt > -\infty\,,
\end{equation*}
where $y(\cdot)$ is the solution to the Cauchy problem corresponding to $y_0$ and to the control function
$u(\cdot)$.
As it is well known, one is led to solving matrix/operator {\em inequalities}, rather than 
(matrix/operator) Riccati {\em equations};
in addition, when $T=+\infty$ the analysis moves from the {\em time domain} to the {\em frequency domain}.

It is far beyond the scope of this article to go into any detail of this problem, whose key result is
the celebrated Kalman-Popov-Yakubovich Lemma.
However, we find it fitting to include in this synopsis a mention of the above described scenario, owing to the variety of its motivations and applications -- some of them dating back to much earlier times than the sixties --, as well as of the extensions of Yakubovich's results to the PDE realm that have been attained in the nineties. 
We refer the interested reader to \cite{pan-KPY_1997} and \cite{bu-pan-scl_2000}, which include historical comments and pertinent references.

\subsubsection{Operator Riccati equations}
During the seventies, the theories of the LQ-problem and of Riccati equations is extended to the infinite dimensional setting.
The abstract formulation of initial-boundary value problems (IBVP) for PDE in bounded domains, in the presence of {\em distributed} control, 
brings about control systems of the usual form $y'=Ay+Bu$, where 
$A$ is the infinitesimal generator of a $C_0$-semigroup in a certain Banach space $Y$, while 
$B$ is a {\em bounded} operator from a (control) space $U$ to $Y$; see \cite{bddm}.
(The same framework encompasses also other kinds of differential problems such as, e.g., delay systems.)
The first contributions to the study of Riccati equations with unbounded operator coefficients
(still, with $B$ bounded) are due to the works of Da Prato \cite{daprato_1973}, Tartar, 
Curtain and Pritchard, Zabczyk; see \cite[Part~IV, Ch.~4]{zabczyk-book} and \cite{bddm}.

\subsubsection{Riccati equations with unbounded $B$} %control operator}
The proof of well-posedness performed by Da Prato in \cite{daprato_1973} paves the way for subsequent direct studies of Riccati equations with an {\em unbounded} (control) operator $B$.
These extensions were initially focused on mixed problems for {\em parabolic} PDE.
Hence, they could exploit the regularity properties brought about by the analytic semigroup that describes the free dynamics.
Although well known to those acquainted with the subject, it is important to emphasize
that the transition to infinite-dimensional systems describing PDE
with {\em boundary} or {\em point} control (as opposed to PDE with {\em distributed} control), 
has drastically impacted on the mathematical analysis of the corresponding Riccati equations.  
Whatever approach and method of proof is chosen, the actual meaning of the term $B^*P$
which appears in the quadratic term of the equation -- or lack thereof --
is the central issue that must be tackled.  
It is here that the parabolic and the hyperbolic analyses split apart.

Even more notably, when a Riccati theory
in connection with optimal boundary control of hyperbolic-like PDE is ultimately attained, 
greater mathematical challenges come from the trace regularity results that are essential 
in order to invoke it, with respective proofs on an {\em ad hoc} basis.

We refer the reader to \cite{las-trig-lncis} for a useful overview of the main results on the subject
until the early nineties, and to \cite{bddm} for a broader picture.
The in-depth monograph \cite{las-trig-redbooks} goes into much more detail; 
the notes at the end of any chapter provide a roadmap as well as bibliographical information.
Principal contributions to the Riccati theory for parabolic (and parabolic-like) PDEs are
due to Balakrishnan (1977), Lasiecka and Triggiani (1983),
Pritchard and Salamon (1984), Flandoli (1984), Da Prato and Ichikawa (1985).
The unifying functional-analytic approach such as the one of Salamon \cite{salamon_1987} appears upweighed by
%As for 
the Riccati theory developed specifically for hyperbolic (hyperbolic-like, including Petrovski\u{\i}-type) PDE.
We recall explicitly the direct study of differential Riccati equations by Da Prato, Lasiecka and Triggiani
\cite{dap-las-trig_1986}; the most conclusive results on algebraic Riccati equations 
are due to Flandoli (1988) and Barbu (2000), both joint with Lasiecka and Triggiani.

\subsubsection{A glimpse of peculiarities and challenges relevant to hyperbolic-like PDE}
The hyperbolic character of the dynamics -- in the present work, of one component of the coupled PDE system --, combined with the unboundedness of the control operator, % produces 
can raise obstacles to the well-posedness of the Riccati equations 
corresponding to the associated optimal control problems.
It will suffice to highlight that under the {\em abstract trace regularity} assumption or 
{\em admissibility condition}
\begin{equation*}
\exists C_T>0 \colon \quad \int_0^T \|B^*e^{A^*t}x\|_U^2\le C_T \|x\|_Y^2 \qquad \forall x\in Y
\end{equation*}
which is characteristic of hyperbolic-like dynamics, 
\begin{itemize}
\item
well-posedness of differential Riccati equations holds true provided the observation operator
$R$ has appropriate {\em smoothing properties};

\item
given the optimal cost operator $P$ for the infinite time horizon problem, it turns out that the gain operator $B^*P$ that occurrs in the algebraic Riccati equation may be even {\em not densely defined}.
(Then, appropriate extensions of $B^*P$ are called for.)
\end{itemize}
A simple illustration of the latter weak point is given by Weiss and Zwart \cite{w-z_1988}.
This work exhibits a first-order hyperbolic PDE (in one space dimension), with {\em point} control;
given a certain quadratic functional, the optimal cost operator $P$ is computed explicitly, 
and then it is shown that $B^*$ is intrinsically not defined on $Py$, $y\in \cD(A)$.
In this connection we remark that, to the best of our knowledge, the question as to whether there 
actually exist examples of hyperbolic PDE with boundary (rather than point) control which give rise to the same `failure', is open.

We like to point out that in the case of systems of hyperbolic-parabolic PDEs (with boundary control) which fulfil the requirements of our Assumptions~\ref{a:ipo_1} (or Assumptions~\ref{a:ipo_2}, when time
varies in the positive half line), the highlighted difficulties are overcome, as the theory of the Riccati 
equations devised in \cite{abl_2005,abl_2013} and in the present work shows.

\smallskip
The (full) Bolza problem -- that is a finite time horizon problem where the quadratic functional to be minimized includes a penalization of the state at the final time $T$ -- 
renders the study of well-posedness of Riccati equations more intricate.
A technical discussion of this subject is beyond the scope of this synopsis. 
We refer the reader to \cite{las-trig-redbooks} and its references.

We just remark that the case of functionals of Bolza type is not addressed in \cite{abl_2005}; 
the Bolza problem for the class of systems characterized by the Assumptions~\ref{a:ipo_1}
is indeed open.

\subsubsection{The non-autonomous case}
The case of time dependent (operator) coefficients in the Riccati equation is considerably more difficult.
The major contributions, that all pertain to the parabolic case, are due to Acquistapace and Terreni,
in part jointly with Flandoli; see \cite{acq-fla-ter_1991}, \cite{acq-ter_2000}.

\subsubsection{Approaching the present scenario}
We turn now our attention to PDE systems comprising two (or more) evolutionary PDE of different type, which further display a {\em strong coupling} and are subject to boundary/interface/point control actions.
%
% It is clear that the theories devised specifically for parabolic or hyperbolic cases are useless. 
% 
We owe to Avalos and Lasiecka the understanding of the role played by certain (local in time) estimates for the kernel $e^{At}B$ for the purpose of boundedness of the gain operator $B^*P(t)$, even in the absence of
analyticity of the semigroup $e^{At}$.
It was the theoretical study of the quadratic optimal control problem for a PDE model of an acoustic-structure interaction, successfully carried out in \cite{avalos-las_1996}, to 
motivate the subsequent introduction of the class of control systems which yield 
the said {\em singular estimates}.
The theory relative to the corresponding (differential as well as algebraic) Riccati equations is well recognized:
after the former works \cite{las-cbms}, \cite{las-trento}, \cite{las-trig-se-1,las-trig-se-2},
a great deal of attention has been devoted to the Bolza problem by Lasiecka and Tuffaha
(see \cite{las-tuff-theory_2008}, \cite{las-tuff-theory_2009}, \cite{tuffaha-aa_2013}).
It should be pointed out that the question of uniqueness of Riccati equations appear to be
left out of the aforementioned works, with the exception of \cite{tuffaha-aa_2013}.

\smallskip
The class of control systems at hand here -- introduced by these authors, jointly with Lasiecka, in \cite{abl_2005} -- is instead characterized by more elaborate regularity properties of the adjoint operator $B^*e^{A^*t}$, recorded in the next section as Assumptions~\ref{a:ipo_1} (these assume the slightly strenghtened form of Assumptions~\ref{a:ipo_2} for the problem on the infinite time interval). 
These assumptions have proved effective in describing significant systems of strongly coupled hyperbolic/parabolic PDE which do not fall in the previous class, as has been pointed out already.
Within this framework, complete theories of the quadratic optimal control problem on a finite and infinite time horizon have been devised in \cite{abl_2005} and \cite{abl_2013}, which bring about existence for the corresponding differential and algebraic Riccati equations, 
with the distinctive feature that $B^*P$ is defined on a {\em dense} set.
An insightful comparison among all the mentioned abstract classes, along with their
respective Riccati theories, is provided by \cite[Section~2.1]{abl_2005}.

Thus, focus of the present work is on uniqueness of Riccati equations, an issue which was not discussed and is absent in the aforesaid works \cite{abl_2005} and \cite{abl_2013}. 
While benefiting from the already established existence, our proofs of uniqueness tackle (and overcome)
the challenges brought about by the peculiar regularity properties of the composite dynamics, 
singled out by the Assumptions~\ref{a:ipo_1} and \ref{a:ipo_2}.

\subsection{Outline of the paper}
The structure of the paper is outlined readily.
In the next Subsection~\ref{s:framework} we give the statements of our uniqueness results, 
that are Theorem~\ref{t:dre-uniqueness} and Theorem~\ref{t:are-uniqueness}, after
having recalled the framework and the core statements of the theories of the LQ-problem devised in \cite{abl_2005} and \cite{abl_2013}. 

In Section~\ref{s:uniqueness_1} we present a first proof of Theorem~\ref{t:dre-uniqueness}.
This is preceded by Lemma~\ref{l:DtoI}, which establishes two integral forms of the
differential Riccati equation.
An integral form of the algebraic Riccati equation is derived as well, in Lemma~\ref{l:are-integral}.

Section~\ref{s:uniqueness_2} provides a second proof of Theorem~\ref{t:dre-uniqueness} and the
proof of Theorem~\ref{t:are-uniqueness}.
Instrumental to the proofs are Lemma~\ref{l:id-fondamentale} and Lemma~\ref{l:closedloop} for the former result, and Lemma~\ref{l:id_fond_infty} along with Lemma~\ref{l:closedloopinfty} for the latter.
These lemmas establish the fundamental identities and discuss certain built closed loop equations.

In Appendix~\ref{a:appendix} we gather several regularity results (some old, some new)
which are used throughout the paper.

% SECTION 1 (Richiami)

\section{Abstract framework, main results} \label{s:framework}
\subsection{The LQ problem: abstract dynamics and setting}
Let $Y$ and $U$ be two separable Hilbert spaces, the {\em state} and {\em control} space, respectively.
We consider the abstract (linear) control system $y'=Ay+Bu$ and the corresponding Cauchy problem

\begin{equation}\label{e:state-eq}
\begin{cases}
y'(t)=Ay(t)+Bu(t)\,, & \quad 0\le t<T
\\[2mm]
y(0)=y_0\in Y\,, & 
\end{cases}
\end{equation}
under the following basic Assumptions.
\begin{assumptions}[\bf Basic Assumptions] \label{a:ipo_0} 
Let $Y$, $U$ be separable complex Hilbert spaces. 
\begin{itemize}
\item
The closed linear operator $A:\cD(A)\subset Y \to Y$ is the infinitesimal generator of a strongly continuous semigroup $\{e^{At}\}_{t\ge 0}$ on $Y$;
\item 
$B\in \cL(U,[\cD(A^*)]')$.
\end{itemize}

\end{assumptions} 

\begin{remarks}
\begin{rm}
The apparent weakness of the basic assumptions on the operators $A$ and $B$ which characterize the 
initial/boundary value problems described by the abstract equation in \eqref{e:state-eq} 
is a reflection of the modeling of boundary control problems for systems of coupled hyperbolic/parabolic PDEs. 
It is worth recalling its most prominent features:
(i) first, the control operator $B$ will {\em not} be bounded from the control space $U$ into the state space $Y$; (ii) secondly, the semigroup $e^{At}$ will {\em not} be analytic.
\\
We remind the reader that (i) is intrinsic to the mathematical modeling of control actions on the boundary of the domain (or on some part of it), as first shown in \cite{fattorini_1968}.
We also note that the presence of control actions concentrated on points (in $1$-D) or on curves (in $2$-D) in the interior of domain results in unboundedness of the control operator as well;
illustrations of both situations are found in \cite{las-trig-lncis}, \cite{bddm} and \cite{las-trig-redbooks}.
As for the simple requirement in (ii), it is a natural feature of composite dynamics comprising solely a parabolic component: analiticity of the overall semigroup should not be expected.
\end{rm}
\end{remarks}

\smallskip
Thus, given $y_0\in Y$, the Cauchy problem \eqref{e:state-eq} possesses a unique {\em mild} solution given by 
\begin{equation} \label{e:mild} 
y(t)= e^{At}y_0+\int_0^t e^{A(t-s)}Bu(s)\,ds\,, \qquad t\in [0,T)\,,
\end{equation}
where
\begin{equation*}
L\colon u(\cdot) \longrightarrow (Lu)(t) :=\int_0^t e^{A(t-s)}Bu(s)\,ds
\end{equation*}
is the {\em input-to-state mapping}, that is the operator which associates to any control function 
$u(\cdot)$ the solution to the Cauchy problem \eqref{e:state-eq} with $y_0=0$, and \eqref{e:mild} 
makes sense at least in the extrapolation space $[\cD(A^*)]'$; 
see \cite[\S~0.3, p.~6, and Remark~7.1.2, p.~646]{las-trig-redbooks}.
\\
We will use the notation $L$ throughout the paper. 
We point out here the definition \eqref{e:convolution-in-st} of the operator $L_s$, which will occur 
later; the symbol $L_0$, in place of $L$, is avoided for the sake of simplicity.

\smallskip
To the state equation \eqref{e:state-eq} we associate the quadratic functional
\begin{equation} \label{e:cost}
J(u)=\int_0^T \left(\|Ry(t)\|_Z^2 + \|u(t)\|_U^2\right)dt\,, 
\end{equation}
where $Z$ is a third separable Hilbert space -- the so called observation space (possibly,
$Z\equiv Y$) -- and at the outset the {\em observation} operator $R$ simply satisfies 
\begin{equation}\label{e:basic-for-r}
R\in \cL(Y,Z)\,.
\end{equation}
The formulation of the optimal control problem under study is classical. %formulated in the classical way.
The adjectives finite or infinite time horizon problem refer to the cases $T<+\infty$ or $T=+\infty$,
respectively. 
\begin{problem}[\bf The optimal control problem] \label{p:problem-0}
Given $y_0\in Y$, seek a control function $u\in L^2(0,T;U)$ which minimizes the
cost functional \eqref{e:cost}, where $y(\cdot)=y(\cdot\,;y_0,u)$ is the solution to
\eqref{e:state-eq} corresponding to the control function $u(\cdot)$ (and with initial state $y_0$)
given by \eqref{e:mild}.
\end{problem}

It is well known that aiming at solving Problem~\ref{p:problem-0}, certain
principal facts need to be ascertained, beside the existence of a unique optimal pair 
$(\hat{u}(\cdot,s;y_0),\hat{y}(\cdot,s;y_0))$ (which is readily established by using classical variational arguments); namely, 
\begin{itemize}
\item[-]
that the optimal control $\hat{u}(t)$ admits a (pointwise in time) {\em feedback} representation,
in terms of the optimal state $\hat{y}(t)$;
\item[-]
that the optimal cost operator $P(t)$ ($P$, when $T=+\infty$) solves the corresponding 
Differential (Algebraic) Riccati equation; thus, the issue of well-posedness of the DRE (ARE)
arises, requiring  
\item[-]
that a meaning is given to the gain operator $B^*P(t)$ ($B^*P$) on the state space $Y$ (by means of extensions, or -- and this will be the case here --, as a {\em bounded} operator
on a dense subset of $Y$).
\end{itemize}

% RICHIAMI: ORIZZONTE FINITO

\subsection{Theoretical results: finite and infinite time horizon problems}
We begin by recalling the theory of the LQ-problem on a finite time interval developed in 
\cite{abl_2005}. 
This theory pertains to the class of control systems -- introduced in the very same
\cite{abl_2005} -- whose dynamics, control and observation operators are subject to the following assumptions.

\begin{assumptions}[\bf Finite time horizon case] \label{a:ipo_1} 
Let $Y$, $U$ and $Z$ be separable complex Hilbert spaces, and let $T>0$ be given.
The pair $(A,B)$ (which describes the state equation \eqref{e:state-eq}) fulfils 
Assumptions~\ref{a:ipo_0}, with the additional property $A^{-1}\in \cL(Y)$, while the 
observation operator $R$ (which occurs in the cost functional \eqref{e:cost}) satisfies the 
basic condition~\eqref{e:basic-for-r}.

The operator $B^*e^{A^*t}$ can be decomposed as 
\begin{equation} \label{e:key-hypo} 
B^*e^{A^*t}x = F(t)x + G(t)x\,, \qquad 0\le t\le T\,, \; x\in \cD(A^*)\,,
\end{equation}
where $F(t)\colon Y\longrightarrow U$ and $G(t)\colon \cD(A^*)\longrightarrow U$, $t>0$, are bounded linear
operators satisfying the following assumptions:

\begin{enumerate}
\item[i)] 
there exist constants $\gamma\in (0,1)$ and $N>0$ such that
\begin{equation}  \label{e:keyasF}
\|F(t)\|_{\cL(Y,U)} \le N\,t^{-\gamma}\,, \qquad 0<t\le T\,;
\end{equation}
\item[ii)] 
the operator $G(\cdot)$ belongs to $\cL(Y,L^p(0,T;U))$ for all $p\in [1,\infty)$;
\item[iii)] 
there exists $\epsilon>0$ such that:
\begin{enumerate}
\item[a)] 
the operator $G(\cdot){A^*}^{-\epsilon}$ belongs to $\cL(Y,C([0,T];U))$, 
with
\begin{equation*} 
\sup_{t\in [0,T]}\|G(t){A^*}^{-\epsilon}\|_{\cL(Y,U)} <\infty\,;
\end{equation*}
\item[b)] 
the operator $R^*R$ belongs to $\cL(\cD(A^{\epsilon}),\cD({A^*}^{\epsilon}))$, i.e.
\begin{equation} \label{e:keyasR} 
\|{A^*}^{\epsilon}R^*RA^{-\epsilon}\|_{\cL(Y)} \le c<\infty\,;
\end{equation}
\item[c)] 
there exists $q\in (1,2)$ (depending, in general, on $\epsilon$) such that the map 
$x \longmapsto B^*e^{A^*t}{A^*}^\epsilon x$
% (a priori defined in $\cD({A^*}^{1+\epsilon})$)
has an extension which belongs to $\cL(Y,L^q(0,T;U))$. 
\end{enumerate}
\end{enumerate}

\end{assumptions}

\medskip

\begin{remarks}
\begin{rm}
1. We note that it is assumed at the very outset that $0\in \rho(A)$, i.e. the dynamics operator $A$ is boundedly invertible on $Y$.
It is important to emphasize that this property happens to hold true for an ample variety of
composite systems of hyperbolic-parabolic PDE, such as e.g. thermoelastic systems, structural-acoustics models, fluid-elasticity interactions; see \cite{abl-thermo_2005}, 
\cite{bucci-applicationes}, \cite{bucci-las-fsi_2010,bucci-las-fsi_2011}, \cite{bucci_2020}. 
This allows in particular to define the fractional powers $(-A)^{\alpha}$, $\alpha\in (0,1)$;
see \cite[\S~1.15.1-2]{triebel}, \cite{martinez-etal_1988}, \cite[\S~2.2.2]{lunardi-book}.
(In order to make the notation lighter, we wrote $A^\alpha$ instead of $(-A)^\alpha$; 
the same will happen throughout the paper.)

On the other hand, when $\lambda =0$ is not in the resolvent set of $A$, one can find 
$\omega_0>0$ -- the {\em type} of the semigroup -- such that the translation 
$\hat{A}:=\omega -A$ is a positive operator for any $\omega>\omega_0$; then $\hat{A}$
is boundedly invertible,
and the fractional powers $\hat{A}^\theta$ of $\hat{A}$ are well-defined. 
The extension of the present theory to the case of unstable semigroups $e^{At}$ is particularly relevant in the infinite time horizon case ($T=+\infty$). 
It would certainly require a tedious series of technical changes and {\em is so far lacking}. 

\smallskip
\noindent
2. If the singular estimate \eqref{e:keyasF} for the component $F$ ({\em cf.}~i) of Assumptions~\ref{a:ipo_1}) 
is shown to hold true in an arbitrarily small right neighbourhood of $t=0$, then
it usually extends to all $t\in (0,T]$, by using semigroup theory.  

\smallskip
\noindent
3. We note that iiia) of the Assumptions~\ref{a:ipo_1} tells us that the `basic' (time) regularity of the $G$ component, that is $G(\cdot)y\in L^p(0,T;U)$ for $y\in Y$ and all {\em finite} summability exponents
$p\ge 1$, improves to $G(\cdot)y\in C([0,T];U)$, when $y\in \cD({A^*}^\epsilon)$. 

\smallskip
\noindent
4. The findings of the work \cite{abl_2005}, summarized in the next Theorem~\ref{t:theory-tfinite}, were actually established under the {\em weaker} regularity assumption

\begin{enumerate}

\item[iiic)'] 
there exists $q\in (1,2)$ such that the map $x \longmapsto B^*e^{A^*t }R^*R {A}^\epsilon x$
has an extension which belongs to $\cL(Y,L^q(0,T;U))$.
\end{enumerate}
Indeed, iiic)' of Assumptions~\ref{a:ipo_1}, combined with iiib), implies readily iiic), as already pointed out in \cite[p.~1401]{abl_2005} (with a reversed notation, though).

However, on one side the present iiic) -- more precisely, the boundary regularity result that (case by case) the control-theoretic condition iiic) translates to -- has been shown over the years to hold true in the case of distinct PDE systems studied in the aforementioned references (\cite{abl-thermo_2005}, 
\cite{bucci-applicationes}, \cite{bucci-las-fsi_2010,bucci-las-fsi_2011}, \cite{bucci_2020}).
On the other side, uniqueness of solutions to the Riccati equations appears to be in need of it:
both within the first proof of Theorem~\ref{t:theory-tfinite} given in the next section
(specifically to perform the estimates which bring about \eqref{e:key-estimate}),
and also to show Lemma~\ref{l:closedloop}, instrumental to the distinct proof of the same result
proposed in Section~\ref{s:uniqueness_2}.
Furthermore, the stronger \eqref{e:da-iiic} -- which is central to the proof of 
Lemma~\ref{l:closedloopinfty} relevant to the infinite time horizon case --
is based upon iiic) of Assumptions~\ref{a:ipo_1}. 
\end{rm}
\end{remarks}

\smallskip
Under the listed Assumptions~\ref{a:ipo_1}, a full solution to the optimal control 
Problem~\ref{p:problem-0}, as detailed by the complex of statements S1.--S6. 
collected in Theorem~\ref{t:theory-tfinite} below, was obtained in \cite{abl_2005}.
These include, in particular, two specific novel features over the {\em parabolic} 
or {\em hyperbolic} cases (\cite{las-trig-lncis}):

\begin{itemize}
\item
the lack of continuity (in time) of the optimal control $\hat{u}(\cdot)$ (see S1.), and 
\item
that the gain operator $B^*P(t)$ is {\em bounded} only on a certain dense subset of $Y$, 
yet not preventing well-posedness of the Differential Riccati Equations corresponding
to the LQ problem. 
\end{itemize}

% TEORIA T FINITO

\begin{theorem}[Finite time horizon theory; cf.~\cite{abl_2005}, Theorem~2.3]
\label{t:theory-tfinite}
With reference to the control problem \eqref{e:state-eq}--\eqref{e:cost}, under
the Assumptions~\ref{a:ipo_1}, the following statements are valid for each $s\in [0,T)$.
\begin{enumerate}
\item[\bf S1.] 
For each $x\in Y$ the optimal pair $(\hat{u}(\cdot,s;x),\hat{y}(\cdot,s;x))$ satisfies 
\begin{equation*}
\hat{y}(\cdot,s;x)\in C([s,T];Y), \quad \hat{u}(\cdot,s;x)\in \bigcap_{1\le p<\infty}L^p(s,T;U).
\end{equation*}
\item[\bf S2.] 
The linear bounded (on $Y$) operator $\Phi(t,s)$, defined by 
\begin{equation} \label{statop} 
\Phi(t,s)x = \hat{y}(t,s;x) = e^{A(t-s)}x + [L_s \hat{u}(\cdot,s;x)](t)\,,
\quad s\le t\le T\,, \; x\in Y\,,
\end{equation}
is an evolution operator, i.e.
\begin{equation*}
\Phi(t,t)=I_Y\,, \qquad 
\Phi(t,s)=\Phi(t,\sigma)\Phi(\sigma,s) \quad \textrm{for $s\le \sigma\le t\le T$.}
\end{equation*}
\item[\bf S3.] 
For each $t\in [0,T]$ the operator $P(t)\in \cL(Y)$, defined by 
\begin{equation} \label{opRic}
P(t)x = \int_t^T e^{A^*(\tau-t)} R^*R \Phi(\tau,t)x\ d\tau\,, 
\qquad x\in Y,
\end{equation}
is self-adjoint and positive; it belongs to $\cL(Y,C([0,T];Y))$ and is such 
that
\begin{equation*}
(P(s)x,x)_Y = J_s(\hat{u}(\cdot,s;x),\hat{y}(\cdot,s;x)) \qquad \forall 
s\in [0,T]\,.
\end{equation*}
\item[\bf S4.] 
The gain operator $B^*P(\cdot)$ belongs to $\cL(\cD(A^{\varepsilon}),C([0,T];U))$ and 
the optimal pair satisfies for $s \le t \le T$ 
\begin{equation} \label{feedb}
\hat{u}(t,s;x) = -B^*P(t)\hat{y}(t,s;x) \qquad \forall x\in Y.
\end{equation}
\item[\bf S5.] 
The operator $\Phi(t,s)$ defined in \eqref{statop} satisfies for $s<t\le T$:
\begin{equation} \label{dfids} 
\frac{\partial\Phi}{\partial s}(t,s)x = -\Phi(t,s)(A-BB^*P(s))x \in
L^{1/\gamma}(s,T;[\cD({A^*}^{\varepsilon})]')
\end{equation}
for all $x\in \cD(A)$, and
\begin{equation}  \label{dfidt} 
\frac{\partial\Phi}{\partial t}(t,s)x = (A-BB^*P(t))\Phi(t,s)x \in C([s,T],[\cD(A^*)]') 
\end{equation}
for all $x\in \cD(A^{\varepsilon})$.
\item[\bf S6.] 
The operator $P(t)$ defined by \eqref{opRic} satisfies the following (differential) Riccati equation in
$[0,T)$:

\begin{equation} \label{e:DRE}
\begin{cases}
\frac{d}{dt}\,(P(t)x,y)_Y + (P(t)x,Ay)_Y + (Ax,P(t)y)_Y + (Rx,Ry)_Z  
\\[1mm] 
\myspace\quad - (B^*P(t)x,B^*P(t)y)_U = 0 \qquad\qquad\forall\, x,y\in \cD(A)
\\[2mm]
P(T)=0\,. 
\end{cases}
\end{equation}

\end{enumerate}
\end{theorem}

\bigskip
Among the fundamental conclusions of Theorem~\ref{t:theory-tfinite} is assertion S6.,
namely the property that the optimal cost operator $P(\cdot)$ defined in \eqref{opRic}
does solve the differential Riccati equation (DRE) corresponding to 
Problem~\eqref{e:state-eq}-\eqref{e:cost}.
That $P(\cdot)$ is actually the {\em unique} solution to the DRE \eqref{e:DRE}, at least 
within an appropriate class of operators, is an issue which was not explicitly dealt with
in the paper \cite{abl_2005}. 

Thus, in order to render the finite time horizon theory devised in \cite{abl_2005} complete, 
we complement assertion S6.~of Theorem~\ref{t:theory-tfinite} about existence of solutions
to the DRE \eqref{e:DRE} with the (novel) achievement of {\em uniqueness}, thereby concluding the proof of 
well-posedness of the DRE.
As we will see, uniqueness is meant within a suitable class -- that is class $\cQ_T$ in \eqref{e:class-tfinite} below -- of linear, bounded, self-adjoint operators also meeting an additional requirement, which is consistent with the regularity property displayed by the gain operator in assertion S4. above.

% THM: UNICITA' DRE

\begin{theorem}[\bf Uniqueness for the DRE] 
\label{t:dre-uniqueness}
With reference to the control problem \eqref{e:state-eq}--\eqref{e:cost}, let
the Assumptions~\ref{a:ipo_1} hold.
Then,
the differential Riccati equation \eqref{e:DRE} has a unique solution within the class
\begin{equation}\label{e:class-tfinite}
\begin{split}
\cQ_T &=\big\{Q\in C([0,T];\cL(Y))\colon  \; Q(t)=Q(t)^*\ge 0\,, %\; t\in [0,T]\,, 
\; Q(T)=0\,, 
\\[1mm]
& \myspace\myspace
B^*Q(\cdot)\in \cL(\cD(A^{\epsilon}),C([0,T];U))\big\}\,.
\end{split}
\end{equation}

\smallskip
\noindent
The optimal cost operator $P(\cdot)$ defined by \eqref{opRic} is consequently that solution.

\end{theorem}

\begin{remark}
\begin{rm}
We give two distinct proofs of Theorem~\ref{t:dre-uniqueness}: a first one in Section~\ref{s:uniqueness_1}
and a second one in Section~\ref{s:uniqueness_2}.
Both proofs utilize the link between the differential and the {\em integral} forms
of the Riccati equation, as clarified in the beginning of the next section.
The standing Assumptions~\ref{a:ipo_1} play a key role in both proofs, as expected, with the trickier iiic) influencing them in a decisive and crucial way.
A major technical step in the longer second proof of uniqueness is the derivation of a {\em fundamental identity} which is classical in control theory, stated as Lemma~\ref{l:id-fondamentale}.
\end{rm}
\end{remark}

\smallskip
In the infinite time horizon case -- i.e., when $T=+\infty$ in \eqref{e:cost} -- 
appropriate requirements on the decay (as $t\to +\infty$) of the semigroup $e^{At}$ as well as
of the component $F(t)$ involved in the decomposition of the operator $B^*e^{A^*t}$ are introduced, which both appear very natural; see \eqref{e:stability} 
and \eqref{e:asymptotics} (the latter being (i)' of Assumptions~\ref{a:ipo_2})
below, respectively.

Interestingly, as a consequence of the aforesaid asymptotic behaviour, the requirements on 
the $L^p$ (in time) regularity of the component $G(\cdot)$ as well as of the operator 
$B^*e^{A^*\cdot}{A^*}^\epsilon$ will need to hold only on a certain bounded interval $[0,T]$,
rather than on the entire half-line $(0,\infty)$.
\\
For the sake of completeness and the reader's convenience, the hypotheses pertaining to the infinite time horizon are wholly recorded below.

\begin{assumptions}[\bf Infinite time horizon case]\label{a:ipo_2} 
Let $Y$, $U$ and $Z$ be separable complex Hilbert spaces, and let the basic Assumptions~\ref{a:ipo_0} be valid, with the additional property that the $C_0$-semigroup $e^{At}$ is exponentially stable on $Y$,
$t\ge 0$; namely, there exist constants $M\ge 1$ and $\omega>0$ such that 
\begin{equation} \label{e:stability}
\|e^{At}\|_{\cL(Y)} \le M \,e^{-\omega t} \qquad \forall t\ge 0\,.
\end{equation}
Then in particular, $A^{-1}\in \cL(Y)$.

The operator $B^*e^{A^*t}$ admits the decomposition \eqref{e:key-hypo}, where  
$F(t)\colon Y\longrightarrow U$, $t \ge 0$, is a bounded linear operator such that 
\begin{enumerate}
\item[i)'] 
there exist constants $\gamma\in (0,1)$ and $N, \eta>0$ such that 
\begin{equation} \label{e:asymptotics}
\|F(t)\|_{\cL(Y,U)} \le N\,t^{-\gamma}\,e^{-\eta t}\qquad \forall t>0\,,
\end{equation}
\end{enumerate}
while ii)-iiia)-iiib)-iiic) of the Assumptions~\ref{a:ipo_1} on the (linear, bounded) component 
$G(t)\colon \cD(A^*)\longrightarrow U$, $t\ge 0$, hold true for some $T>0$.
\end{assumptions}

We note that the functional \eqref{e:cost} with $T=+\infty$ makes sense at least
for $u\equiv 0$.
This again in view of the exponential stability of the semigroup $e^{At}$ 
(\eqref{e:stability} of Assumptions~\ref{a:ipo_2}), which combined with \eqref{e:basic-for-r} ensures 
$Ry(\cdot,y_0;0)\in L^2(0,\infty;Y)$.
\\
(The analysis carried out in the present paper easily extends to more general quadratic functionals, like
\begin{equation*}
J(u)=\int_0^\infty \big(\|Ry(t)\|_Z^2 + \|\tilde{R}u(t)\|_U^2\big)dt\,,
\end{equation*}
provided $\tilde{R}$ is a coercive operator in $U$. 
We take $\tilde{R}=I$ just for the sake of simplicity and yet without loss of generality.)

% TEORIA T INFINITO

\begin{theorem}[Infinite time horizon theory; cf.~\cite{abl_2013}, Theorem~1.5] 
\label{t:theory-tinfinite}
Under the Assumptions~\ref{a:ipo_2}, the following statements are valid.
\begin{enumerate}
\item[\bf A1.] 
For any $y_0\in Y$ there exists a unique optimal pair 
$(\hat{u}(\cdot),\hat{y}(\cdot))$ for Problem~\eqref{e:state-eq}-\eqref{e:cost}, 
which satisfies the following regularity properties
\begin{align*}
& \hat{u}\in \bigcap_{2\le p<\infty} L^p(0,\infty;U)\,,
\\
& \hat{y}\in C_b([0,\infty);Y) \cap \big[\bigcap_{2\le p< \infty} L^p(0,\infty;Y)\big]\,.
\end{align*}
\item[\bf A2.] 
The family of operators $\Phi(t)$, $t\ge 0$, defined by 
\begin{equation} \label{e:optimal-state-semigroup}
\Phi(t)y_0 :=\hat{y}(t)=y(t,y_0;\hat{u})
\end{equation}
is a $C_0$-semigroup on $Y$, $t\ge 0$, which is exponentially stable.
\item[\bf A3.] 
The operator $P\in\cL(Y)$ defined by 
\begin{equation} \label{e:optimal-cost-op}
Px := \int_0^\infty e^{A^*t}R^*R \Phi(t)x\, dt, \qquad x\in Y,
\end{equation}
is the {\em optimal cost operator}; $P$ is (self-adjoint and) non-negative.
\item[\bf A4.] 
The following ({\em pointwise in time}) feedback representation of the optimal control 
is valid for any initial state $y_0\in Y$:
\begin{equation*} %\label{e:feedback}
\hat{u}(t) = - B^*P \hat{y}(t) \qquad \textrm{for a.e. $t\in (0,\infty)$},
\end{equation*}
where the gain operator satisfies $B^*P\in \cL(\cD(A^\epsilon),U)$
(that is, it is just densely defined on $Y$ and yet it is bounded on $\cD(A^\epsilon)$).
\item[\bf A5.] 
The infinitesimal generator $A_P$ of the (optimal state) semigroup $\Phi(t)$ 
defined in \eqref{e:optimal-state-semigroup} coincides with the operator $A(I-A^{-1}BB^*P)$;
more precisely,
\begin{align*}
& A_P\equiv A(I-A^{-1}BB^*P)\,, \\[1mm]
& \cD(A_P)\subset \big\{x\in Y: \; 
x-A^{-1}B\,B^*Px\in \cD(A) \big\}\,.
\end{align*}
\item[\bf A6.] 
The operator $e^{At}B$, defined in $U$ and a priori with values in $[\cD(A^*)]'$,
is such that
\begin{equation}\label{e:tricky-regularity}
e^{\delta\cdot}e^{A\cdot}B\in \cL(U,L^p(0,\infty;[\cD({A^*}^{\epsilon})]') 
\quad 
\forall p\in [1,1/{\gamma})
\end{equation}
for all $\delta\in [0,\omega\wedge \eta)$; 
almost the very same regularity is inherited by the operator $\Phi(t)B$:
\begin{equation*}\label{e:low-regularity}
e^{\delta\cdot}\Phi(\cdot)B\in \cL(U,L^p(0,\infty;[\cD({A^*}^{\epsilon})]') 
\quad 
\forall p\in [1,1/{\gamma})\,,
\end{equation*}
with $\delta>0$ sufficiently small.
\item[\bf A7.] 
The optimal cost operator $P$ defined in \eqref{e:optimal-cost-op} is a solution to 
the algebraic Riccati equation (ARE) corresponding to Problem~\eqref{e:state-eq}-\eqref{e:cost},
that is \eqref{e:ARE}.
The ARE reads as
\begin{equation*}
(A^*Px,z)_Y+(x,A^*Pz)_Y-(B^*Px,B^*Pz)_U+(Rx,Rz)_Z=0  
\end{equation*}
when $x,z\in \cD(A_P)$.

\end{enumerate}

\end{theorem}

\smallskip
In order to render the infinite time horizon theory devised in \cite{abl_2013} complete, we complement assertion A7. of Theorem~\ref{t:theory-tinfinite} about existence of solutions to the ARE \eqref{e:ARE} corresponding to Problem~\eqref{e:state-eq}-\eqref{e:cost},
with the achievement of {\em uniqueness}, thereby concluding the proof of 
well-posedness of the ARE.

% THM: UNICITA' ARE

\begin{theorem}[\bf Uniqueness for the ARE] 
\label{t:are-uniqueness}
Consider the optimal control problem~\eqref{e:state-eq}-\eqref{e:cost}, with $T=+\infty$, 
under the Assumptions~\ref{a:ipo_2}.
Then, 
the algebraic Riccati equation \eqref{e:ARE} has a unique solution $P$ within the class $\cQ$ 
defined as follows:
\begin{equation} \label{e:classeQ}
\cQ :=\big\{ Q\in \cL(Y)\colon Q=Q^*\ge 0\,, \; B^*Q\in \cL(\cD(A^\epsilon),U)\big\}\,.
\end{equation} 

\smallskip
\noindent
The optimal cost operator $P$ defined by \eqref{e:optimal-cost-op} is consequently that 
solution.

\end{theorem}

\begin{remark}
\begin{rm}
As in the finite time horizon case, the (linear, bounded, self-adjoint) operators that
belong to the class $\cQ$ in \eqref{e:classeQ} are characterized by
a requirement that is consistent with the regularity property displayed by the gain operator in assertion A4. of Theorem~\ref{t:theory-tinfinite}.
\end{rm}
\end{remark}

% SECTION 2 (Dim. breve)

\section{A first proof of uniqueness for the DRE} \label{s:uniqueness_1}
In this Section we derive integral forms of both the differential and algebraic Riccati equations,
and present a first proof of Theorem~\ref{t:dre-uniqueness}.
The argument employed in this proof is pretty standard: the difference % $Q(t)=P(t)-P_1(t)$ 
between another possible solution $P_1(t)$ to the DRE and the Riccati operator $P(t)$ 
is shown -- in a series of steps -- to be identically zero on the interval $[0,T]$.
The final goal is attained fully exploiting the Assumptions~\ref{a:ipo_1} and more specifically iiic),
having taken as a starting point the integral form of the DRE.
 
By contrast, in the next Section~\ref{s:uniqueness_2} a unified approach and method of proof will prove effective in showing uniqueness for both cases. 

% SUBSECTION 1.1

\subsection{Finite time interval, differential Riccati equations}
In this subsection we make reference to the optimal control problem \eqref{e:state-eq}--\eqref{e:cost},
with $T<+\infty$.
We address the issue of uniqueness of solutions to the Cauchy problem \eqref{e:DRE} for the Riccati equation corresponding to problem \eqref{e:state-eq}--\eqref{e:cost}, under the Assumptions~\ref{a:ipo_1}.

We begin by relating the differential form \eqref{e:DRE} of the Riccati equation to
an integral form of it, which in turn can be further interpreted.

% Lemma DRE ===> IRE
 
\begin{lemma}[Integral forms of the Riccati equation]  %[DRE $\Longrightarrow$ IRE] 
\label{l:DtoI} 
Let $\cQ_T$ be the class defined in \eqref{e:class-tfinite}, and let 
$Q(\cdot)\in \cQ_T$ be a solution to the DRE \eqref{e:DRE}. 
Then the following assertions hold true.

\smallskip
\noindent
{\bf 1.}
$Q(\cdot)$ solves the integral Riccati equation (in short, IRE), that is 
\begin{equation} \label{e:ire}
\begin{split}
&\big(Q(t)e^{A(t-s)}x,e^{A(t-s)}y\big)_Y 
= (Q(s)x,y)_Y -\int_s^t \big(Re^{A(r-s)}x,Re^{A(r-s)}y\big)_Z\,dr
\\[1mm]
& \myspace + \int_s^t \big(B^*Q(r)e^{A(r-s)}x,B^*Q(r)e^{A(r-s)}y\big)_U\,dr\,,
\end{split}
\end{equation} 
with $0\le s\le t\le T$ and $x,y\in \cD(A^\epsilon)$.

\medskip
\noindent
{\bf 2.}
$B^*Q(\cdot)e^{A(\cdot-s)}\in \cL(Y, L^2(s,T;U))$.

\medskip
\noindent
{\bf 3.} 
The IRE \eqref{e:ire} can be rewritten in the form
\begin{equation} \label{e:ire_bis}
\begin{split}
& \big(e^{A^*(t-s)}Q(t)e^{A(t-s)}x,y\big)_Y 
= (Q(s)x,y)_Y -\int_s^t \big(e^{A^*(r-s)}R^*Re^{A(r-s)}x,y\big)_Y\,dr
\\[1mm]
& \myspace + \int_s^t \big(e^{A^*(r-s)}Q(r)BB^*Q(r)e^{A(r-s)}x,y\big)_Y\,dr\,,
\end{split}
\end{equation} 
valid for any $x,y\in Y$ and with $0\le s\le t\le T$.

\end{lemma}

\begin{proof}
{\bf 1.} 
Let $x,y\in \cD(A)$: then $e^{A\cdot}x$, $e^{A\cdot}y$ are differentiable, and therefore, 
using \eqref{e:DRE}, there exists
\begin{equation*} 
\begin{split}
&\frac{d}{dr}\big(Q(r)e^{A(r-s)}x,e^{A(r-s)}y\big)_Y 
\\[2mm]
& \quad 
= -\big(Q(r)e^{A(r-s)}x,Ae^{A(r-s)}y\big)_Y- \big(Ae^{A(r-s)}x,Q(r)e^{A(r-s)}y\big)_Y
\\[2mm]
& \qquad - \big(Re^{A(r-s)}x,Re^{A(r-s)}y\big)_Z\,+ \big(B^*Q(r)e^{A(r-s)}x,B^*Q(r)e^{A(r-s)}y\big)_U
\\[2mm]
& \qquad 
+ \big(Q(r)Ae^{A(r-s)}x,e^{A(r-s)}y\big)_Y+ \big(Q(r)e^{A(r-s)}x,Ae^{A(r-s)}y\big)_Y
\\[2mm]
& \quad 
= - \big(Re^{A(r-s)}x,Re^{A(r-s)}y\big)_Z\, + \big(B^*Q(r)e^{A(r-s)}x,B^*Q(r)e^{A(r-s)}y\big)_U\,.
\end{split}
\end{equation*}
Integrating the above identity in $r\in [s,t]$, one readily obtains the IRE \eqref{e:ire}, 
valid for $x,y\in \cD(A)$.
In view of Lemma~\ref{l:gain-on-functions}, the validity of the IRE is extended to all
$x,y\in \cD(A^\epsilon)$ by density.
 
\smallskip
\noindent
{\bf 2.}
By taking now in \eqref{e:ire} $t=T$, $x=y\in \cD(A^\epsilon)$, since $P(T)=0$ we establish
\begin{equation*}
\int_s^T \big\|B^*P(r)e^{A(r-s)}x\big\|_U^2\,dr \le \int_s^T \big\|Re^{A(r-s)}x\big\|_Z^2\,dr \le C\|x\|_Y^2
\end{equation*}
by density.

\smallskip
\noindent
{\bf 3.}
The equivalent form \eqref{e:ire_bis} of the IRE follows in view of {\bf 2.}~and by density.

\end{proof}

% PROOF #1.1

\medskip
\noindent
{\em A first proof of Theorem~\ref{t:dre-uniqueness}.}
We follow the proof of Theorem~1.5.3.3 in \cite{las-trig-redbooks}, up to a point.
The subsequent arguments and estimates are driven by 
the distinctive assumptions on the adjoint of the kernel $e^{At}B$, as well as by the different
class of regularity the solutions to the DRE are sought. 

We know already that the optimal cost operator $P(\cdot)$ defined by \eqref{opRic} solves
(the Cauchy problem \eqref{e:DRE} for) the differential Riccati equation, as well as that $P\in \cQ_T$.
Assume there exists another operator in $\cQ_T$, say $P_1(\cdot)$, which solves \eqref{e:DRE},
and set $Q(t):=P_1(t)-P(t)$, $t\in [0,T]$; we aim to prove that $Q(t)\equiv 0$.
By construction $Q(\cdot)\in \cQ_T$. 
By Lemma~\ref{l:DtoI}, both $P_1(\cdot)$ and $P(\cdot)$ satisfy the IRE \eqref{e:ire}.
Then, taking in particular $t=T$, we find that $Q(s)$ satisfies 
\begin{equation} \label{e:first-in-first} 
\begin{split}
(Q(s)x,y)_Y &= -\int_s^T (B^*Q(r)e^{A(r-s)}x, B^*P_1(r)e^{A(r-s)}y)_U\,dr 
\\[1mm]
& \qquad - \int_s^T (B^*P(r)e^{A(r-s)}x,B^*Q(r)e^{A(r-s)}y)_U\,dr\,,
\end{split}
\end{equation}
for any $x,y\in D(A^\epsilon)$.
To render the computations cleaner, set $V(r):=B^*Q(r)$ (that $r$ belongs to $[s,T]$ is omitted here and below, as clear from the context).
Because $Q(\cdot)\in \cQ_T$, it holds $V(r)^*\in \cL(U,[\cD(A^\epsilon)]')$, along with 
\begin{equation*}
\|V(r)^*\|_{\cL(U,[\cD(A^\epsilon)]')}=\|V(r)\|_{\cL(\cD(A^\epsilon),U)} 
\le \|V(\cdot)\|_{\cL(\cD(A^\epsilon),C([s,T];U)}=:c\,.
\end{equation*}
We see that
\begin{equation*}
\left|\langle V(r)^*w,y\rangle_{[\cD(A^\epsilon)]',\cD(A^\epsilon)}\right| 
\le c\,\|w\|_U\|y\|_{D(A^\epsilon)}
\end{equation*}
consequently, as well as that ${A^*}^{-\epsilon} V(r)^*\in \cL(U,Y)$, with
\begin{equation*}
\big|\big({A^*}^{-\epsilon} V(r)^*w,x\big)_Y\big| 
=\left|\langle V(r)^*w,A^{-\epsilon}x\rangle_{[\cD(A^\epsilon)]',\cD(A^\epsilon)} \right| 
\le c \,\|w\|_U\|x\|_Y\,.
\end{equation*}
The same observations apply to $[B^*P_1(r)]^*$ and $[B^*P(r)]^*$, bringing about analogous estimates.

We may now rewrite \eqref{e:first-in-first}  as
\begin{equation*} 
\begin{split}
(Q(s)x,y)_Y & = -\int_s^T (e^{A^*(r-s)}{A^*}^{-\epsilon}[B^*P_1(r)]^* V(r)e^{A(r-s)}x,A^\epsilon y)_Y\,dr
\\[1mm]
& \qquad - \int_s^T (e^{A^*(r-s)}{A^*}^{-\epsilon}V(r)^* B^*P(r)e^{A(r-s)}x,A^\epsilon y)_Y\,dr\,,
\end{split}
\end{equation*}
which tells us that
\begin{equation*} 
\begin{split} 
& {A^*}^\epsilon\int_s^T \Big[e^{A^*(r-s)}{A^*}^{-\epsilon}[B^*P_1(r)]^* V(r)e^{A(r-s)}
\\[1mm]
& \myspace
+ e^{A^*(r-s)}{A^*}^{-\epsilon}V(r)^* B^*P(r)e^{A(r-s)}\Big]x\,dr\,,
\end{split}
\end{equation*}
a priori an element of $[\cD(A^\epsilon)]'$, in fact coincides with $-Q(s)x\in Y$
by the very definition of adjoint operator. %, and then by density.
We deduce 
\begin{equation} \label{e:pre-bstar}
\begin{split} 
Q(s)x 
& = -{A^*}^\epsilon \int_s^T \Big[e^{A^*(r-s)}{A^*}^{-\epsilon}[B^*P_1(r)]^* V(r)e^{A(r-s)}
\\[1mm]
& \myspace
+ e^{A^*(r-s)}{A^*}^{-\epsilon}V(r)^* B^*P(r)e^{A(r-s)}\Big]x\,dr\,,
\end{split}
\end{equation}
valid for every $x\in D(A^\epsilon)$, where, as pointed out above, the right hand side is an element of $Y$.
As $x\in \cD(A^\epsilon)$, $B^*Q(s)x$ is meaningful, and we are allowed to apply $B^*$ to both sides
of \eqref{e:pre-bstar}, thus obtaining 
\begin{equation}\label{e:second-of-first}
\begin{split}
V(s)x &= %B^*Q(s)x =
-B^*{A^*}^\epsilon\int_s^T \Big[e^{A^*(r-s)} 
{A^*}^{-\epsilon}[B^*P_1(r)]^* V(r)e^{A(r-s)} 
\\[1mm]
& \myspace + e^{A^*(r-s)} {A^*}^{-\epsilon}V(r)^* B^*P(r)e^{A(r-s)}\Big]x\,dr\,.
\end{split}
\end{equation}
It is here where iiic) of Assumptions~\ref{a:ipo_1}, that is
\begin{equation*}
\exists \, q\in (1,2)\,, \; C=C(T)>0\colon \quad 
\|B^*e^{A^*(\cdot -s)}{A^*}^\epsilon x\|_{L^q(s,T;U)}\le C\,\|x\|_Y \qquad \forall x\in Y\,,
\end{equation*}
becomes crucially important: indeed, it yields as well
\begin{equation*}
\|[B^*e^{A^*(\cdot -s)}{A^*}^\epsilon]^*g(\cdot)\|_Y \le C\, \|g\|_{L^{q'}(s,T;U)}
\end{equation*}
($q'$ denotes the conjugate exponent of $q$), 
so that in particular
\begin{equation} \label{e:dual-of-tricky}
\|[B^*e^{A^*(\cdot -s)}{A^*}^\epsilon]^*w\|_Y \le C \,(T-s)^{1/q'} \|w\|_U \qquad \forall w\in U\,.
\end{equation}
We return to \eqref{e:second-of-first}, and highlight a few blocks within its right hand side,
as follows:
\begin{equation*}
\begin{split}
V(s)x & = -\int_s^T \big[B^*e^{A^*(r-s)}{A^*}^\epsilon\big]\,
{A^*}^{-\epsilon}[B^*P_1(r)]^* V(r)e^{A(r-s)}x\,dr 
\\[1mm]
& \myspace - \int_s^T\big [B^*e^{A^*(r-s)}{A^*}^\epsilon\big]
{A^*}^{-\epsilon}V(r)^* B^*P(r)e^{A(r-s)}x\,dr\,;
\end{split}
\end{equation*}
multiply next both members by $w\in U$, to find
\begin{equation*}
\begin{split} 
(V(s)x,w)_U & =-\int_s^T \big(V(r)e^{A(r-s)}x,[B^*P_1(r)A^{-\epsilon}]\, 
\big[B^*e^{A^*(r-s)}{A^*}^\epsilon \big]^*w\big)_U\,dr 
\\[1mm]
& \qquad - \int_s^T \big(B^*P(r)e^{A(r-s)}x, [V(r)A^{-\epsilon}]\,
\big[B^*e^{A^*(r-s)}{A^*}^\epsilon\big]^*w\big)_U\,dr\,.
\end{split}
\end{equation*}
We now proceed to estimate either summand in the right hand side, making use of \eqref{e:dual-of-tricky};
this leads to
\begin{equation*}
\begin{split} 
& |(V(s)x,w)_U| 
\\
& \;
\le M \, \|V(\cdot)\|_{\cL(\cD(A^\epsilon),C([s,T];U))} \|x\|_{\cD(A^\epsilon)} 
\|B^*P_1(\cdot)\|_{\cL(\cD(A^\epsilon),C([s,T];U))} 
\|w\|_U (T-s)^{1/q'}
\\
& \quad + M\,\|B^*P(\cdot)\|_{\cL(\cD(A^\epsilon),C([s,T];U))} \|x\|_{\cD(A^\epsilon)}
\|V(\cdot)\|_{\cL(\cD(A^\epsilon),C([s,T];U))}
\|w\|_U (T-s)^{1/q'}\,.
\end{split}
\end{equation*}
Therefore, there exists a positive constant $C$ (depending on $P$ and $P_1$) such that
\begin{equation*}
\big|(V(s)x,w)_U\big| \le C\,\|V(\cdot)\|_{\cL(\cD(A^\epsilon),C([s,T];U))}(T-s)^{1/q'}\|w\|_U \|x\|_{\cD(A^\epsilon)}\,,
\end{equation*}
which establishes
\begin{equation}\label{e:key-estimate}
\|V(s)\|_{\cL(\cD(A^\epsilon),U)} \le C\,\|V(\cdot)\|_{\cL(\cD(A^\epsilon),C([s,T];U))}(T-s)^{1/q'}\,,
\end{equation}
for any $s\in [0,T)$.

The argument is now pretty standard: set $s_0$ such that $(T-s_0)^{1/q'}< 1/C$;
since the estimate \eqref{e:key-estimate} holds true in particular for any $s\in [s_0,T)$,
we have 
\begin{equation*}
\|V(\cdot)\|_{\cL(\cD(A^\epsilon),C([s_0,T];U))} \le C\,(T-s_0)^{1/q'}\,
\|V(\cdot)\|_{\cL(\cD(A^\epsilon),C([s_0,T];U))}
\end{equation*}
which is impossible unless $V(\cdot)\equiv 0$ on $[s_0,T]$. 
Iterating the same argument, in a finite number of steps we obtain $V(s)\equiv 0$ on $[0,T]$. 
This in turn implies, by \eqref{e:first-in-first}, 
\begin{equation*}
(Q(s)x,y)_Y =0 \qquad\forall s\in [0,T]\,, \; \forall x,y\in \cD(A^\epsilon)\,;
\end{equation*}
by density we obtain ($Q(s)x=0$ for any $x\in Y$ first, and then) 
$Q(\cdot)\equiv 0$, that is $P_1(\cdot) \equiv P(\cdot)$, as desired.  

\qed

\subsection{Infinite time interval. Preparatory material} % algebraic Riccati equations}
We turn now our attention to the optimal control problem \eqref{e:state-eq}--\eqref{e:cost},
with $T=+\infty$.
In order to establish a uniqueness result for the corresponding algebraic Riccati equation \eqref{e:ARE}, 
we will employ a different method of proof than the one utilized in the previous subsection for the DRE.
Still, an integral form of the ARE will prove more effective (than its algebraic form)
to accomplish this goal, just like the integral forms of the DRE in Lemma~\ref{l:DtoI} provide fundamental tools for both proofs of Theorem~\ref{t:dre-uniqueness}.
This is the reason why we derive the said integral form of the ARE here.
The following Lemma contributes to the preparatory material
for the forthcoming analysis in Section~\ref{s:uniqueness_2}.
Its proof is not difficult, yet it is explicitly given for the reader's convenience.
 
\begin{lemma}[Integral form of the ARE] \label{l:are-integral} 
Let $\cQ$ be the class defined in \eqref{e:classeQ}, and let $P_1\in \cQ$ be a solution to the
algebraic Riccati equation \eqref{e:ARE}. 
Then, $P_1$ solves the following integral form of the ARE valid for all $x,y\in \cD(A^\epsilon)$: 
\begin{equation} \label{e:are-integral}
\begin{split}
& \hspace{-2mm} \big(P_1e^{A(t-s)}x,e^{A(t-s)}y\big)_Y 
= (P_1 x,y)_Y 
+ \int_s^t \big(B^*P_1e^{A(r-s)}x,B^*P_1e^{A(r-s)}y\big)_U\,dr
\\[1mm]
& \myspace
-\int_s^t \big(Re^{A(r-s)}x,Re^{A(r-s)}y\big)_Z\,dr\,, %\quad x,y\in \cD(A^\epsilon)\,,
\end{split}
\end{equation}
with $0\le s\le t$. 

\end{lemma}

\begin{proof}
Let $P_1\in \cQ$ be a solution to the ARE \eqref{e:ARE}, that we record for the reader's convenience:
\begin{equation*}
(P_1x,Ay)_Y+(Ax,P_1y)_Y-(B^*P_1x,B^*P_1y)_U+(Rx,Ry)_Z=0\,, \quad x, y\in \cD(A)\,.  
\end{equation*}
With $e^{A(t-s)}x$, $e^{A(t-s)}y\in \cD(A)$ in place of $x, y$, and with $0\le s\le t$, the
equation becomes
\begin{equation*}
\begin{split}
& (P_1e^{A(t-s)}x,Ae^{A(t-s)}y)_Y+(Ae^{A(t-s)}x,P_1e^{A(t-s)}y)_Y
\\[1mm]
& \qquad
-(B^*P_1e^{A(t-s)}x,B^*P_1e^{A(t-s)}y)_U +(Re^{A(t-s)}x,Re^{A(t-s)}y)_Z=0\,,
\end{split}
\end{equation*}
that is nothing but
\begin{equation}\label{e:pre-integral}
\begin{split}
&\frac{d}{dt}\big(P_1e^{A(t-s)}x,e^{A(t-s)}y\big)_Y 
= \big(B^*P_1e^{A(t-s)}x,B^*P_1e^{A(t-s)}y\big)_U
\\[1mm]
& \myspace\myspace\quad -\big(Re^{A(t-s)}x,Re^{A(t-s)}y\big)_Z\,,
\quad x, y\in \cD(A)\,.
\end{split}
\end{equation} 
Integrating both sides of \eqref{e:pre-integral} between $s$ and $t$ we attain \eqref{e:are-integral},
initially for any $x,y\in \cD(A)$.
Its validity is then extended to all $x,y\in \cD(A^\epsilon)$ by density, since $P_1\in \cQ$. 
\end{proof}

\medskip
While the integral form \eqref{e:are-integral} of the ARE will constitute the starting point
for the proof of Theorem~\ref{t:are-uniqueness}, it is important to emphasize
the central role of the distinguishing (and improved) regularity properties of the operator 
$B^*e^{A^*\cdot}{A^*}^\epsilon$.
We refer the reader to Appendix~\ref{a:appendix}, where we collected and highlighted several instrumental results, with the aim of displaying their statements in a clear sequence and framework.
See, more specifically, Proposition~\ref{p:2013-P3.2} therein.

% SECTION 3: Dim. tramite l'identita' fondamentale

\section{A unified method of proof of uniqueness for both DRE and ARE}
\label{s:uniqueness_2}
In this Section we provide a second proof of Theorem~\ref{t:dre-uniqueness} and then
show Theorem~\ref{t:are-uniqueness}, thereby settling the question of uniqueness for the differential and algebraic Riccati equations corresponding to the optimal control problem \eqref{e:mild}--\eqref{e:cost}.
We recall from Section~\ref{ss:insight} that the a crucial intermediate step to achieve either goal is an identity which is classical in control theory.

% SubSect DRE

\subsection{Finite time interval, differential Riccati equations}
In this subsection we focus on the optimal control problem \eqref{e:state-eq}--\eqref{e:cost},
with $T<+\infty$, along with the corresponding Riccati equation.
In approaching the second proof of Theorem~\ref{t:dre-uniqueness}, we start by showing the
above-mentioned fundamental identity.
Despite being a standard element in classical optimal control theory, the identity should not be taken for granted in the absence of evident beneficial regularity properties 
of the kernel $e^{At}B$ -- such as analiticity of the semigroup or more generally singular estimates.
Achieving the said equality requires that the Assumptions~\ref{a:ipo_1} are fully exploited.
The delicate, careful computations are carried out in the following Lemma.

% Lemma (Fundamental identity)

\begin{lemma}[Fundamental identity] \label{l:id-fondamentale}
Let $Q\in \cQ_T$ be a solution to the integral Riccati equation \eqref{e:ire}.
With $u\in L^2(s,T;U)$ and $x\in \cD(A^\epsilon)$, let $y(\cdot)$ be the semigroup solution to the state equation in \eqref{e:state-eq} corresponding to $u(\cdot)$, with $y(s)=x$, that is
\begin{equation*}
y(t) = e^{A(t-s)}x + \int_s^t e^{A(t-r)} Bu(r)\,dr=e^{A(t-s)}x + L_su(t)\,, \quad t\in[s,T]\,.
\end{equation*}
Then, the following identity is valid: for $t\in[s,T]$
\begin{equation}\label{e:fundamental}
\begin{split}
& (Q(t)y(t),y(t))_Y - (Q(s)x,x)_Y = -\int_s^t \left[\|Ry(r)\|_Z^2 + \|u(r)\|_U^2\right]\,dr 
\\[1mm]
& \myspace + \int_s^t \|u(r)+B^*Q(r)y(r)\|_U^2 \,dr\,.
\end{split}
\end{equation}

\end{lemma}

\begin{proof}
Assume initially that $u\in L^\infty (s,T;U)$. 
We examine the right hand side of the identity \eqref{e:fundamental}.
For the first term we have
\begin{align*}
-\int_s^t \|Ry(r)\|_Z^2\,dr &= -\int_s^t \big\|Re^{A(r-s)}x\big\|_Z^2\,dr 
-\int_s^t \big\|RL_su(r)\big\|_Z^2\,dr
\\[1mm]
& \quad - 2 \textrm{Re}\int_s^t \big(Re^{A(r-s)}x,RL_su(r)\big)_Z\,dr =: \sum_{j=1}^3 R_j\,.
\end{align*}
We note that each summand $R_j$ makes sense, just considering the space regularity originally singled out in \cite{abl_2005} and here recalled in Proposition~\ref{p:2005-B3}; more specifically, $u\in L^\infty (s,T;U)$ implies $L_s u\in C([s,T];Y)$ by its fourth assertion.
We consider next the remainder 
\begin{equation*}
-\int_s^t \|u(r)\|_U^2\,dr + \int_s^t \|u(r)+B^*Q(r)y(r)\|_U^2 \,dr\,.
\end{equation*}
Computing the square in the second integral, discarding additive inverses and replacing
again the expression of $y(r)$, we get
\begin{align*}
& -\int_s^t \|u(r)\|_U^2\,dr + \int_s^t \|u(r)+B^*Q(r)y(r)\|_U^2 \,dr 
\\[1mm]
& \quad
= 2 \textrm{Re} \int_s^t \big(B^*Q(r)e^{A(r-s)}x,u(r)\big)_U\,dr  
+ 2 \textrm{Re}\int_s^t \big(B^*Q(r)L_s u(r),u(r)\big)_U\,dr
\\[1mm]
& \qquad 
+\int_s^t \big\|B^*Q(r)e^{A(r-s)}x\big\|_U^2\,dr 
+  2 \textrm{Re} \int_s^t (B^*Q(r)e^{A(r-s)}x,B^*Q(r)L_su(r)\big)_U\,dr
\\[1mm]
& \qquad 
+\int_s^t \big\|B^*Q(r)L_su(r)\big\|_U^2\,dr =: \sum_{j=1}^5 C_j\,.
\end{align*}
That each summand $C_j$ makes sense as well is justified by the following observations:
$B^*Q(\cdot)e^{A(\cdot-s)}x\in L^2(s,T;U)$
because of item 2. of Lemma~\ref{l:DtoI};
in addition, since $L^\infty (s,T;U)\subset L^{q'}(s,T;U)$, Lemma~\ref{l:improved} yields the improved regularity $L_s u\in C([s,T];\cD(A^\epsilon))$, which in turn implies 
$B^*Q(\cdot)L_su(\cdot)\in C([s,T];U)$, as shown in Lemma~\ref{l:gain-on-functions}. 

By using the original form \eqref{e:ire} of the integral Riccati equation (IRE), with $x=y$, we find that
\begin{equation}\label{e:prima}
\begin{split}
R_1+C_3& = -\int_s^t \|Re^{A(r-s)}x\|_Z^2\,dr + \int_s^t \big\|B^*Q(r)e^{A(r-s)}x\big\|_U^2\,dr 
\\[1mm]
& =\big(Q(t)e^{A(t-s)}x,e^{A(t-s)}x\big)_Y - (Q(s)x,x)_Y\,.
\end{split}
\end{equation}
Next,
\begin{equation*}
\begin{split}
R_3+C_4+C_1 &= -2 \textrm{Re}\int_s^t \big(Re^{A(r-s)}x,RL_su(r)\big)_Z\,dr
\\[1mm]
& \hspace{-1cm}\quad + 2 \textrm{Re}\int_s^t \big(B^*Q(r)e^{A(r-s)}x,B^*Q(r)L_su(r)\big)_U\,dr 
\\[1mm]
& \hspace{-1cm}\quad + 2 \textrm{Re}\int_s^t \big(B^*Q(r)e^{A(r-s)}x,u(r)\big)_U\,dr
\\[1mm]
& \hspace{-1cm}= -2 \textrm{Re}\int_s^t \Big(R^*Re^{A(r-s)}x,\int_s^r e^{A(r-\sigma)}Bu(\sigma)\,d\sigma\Big)_Y\,dr
\\[1mm]
& \hspace{-1cm}\quad + 2 \textrm{Re} \int_s^t 
\Big\langle Q(r)BB^*Q(r)e^{A(r-s)}x,\int_s^r e^{A(r-\sigma)}Bu(\sigma)\,
d\sigma\Big\rangle_{[\cD(A^\epsilon)]',\cD(A^\epsilon)}\,dr
\\[1mm]
& \hspace{-1cm}\quad + 2 \textrm{Re} \int_s^t \big(B^*Q(r)e^{A(r-s)}x,u(r)\big)_U\,dr\,,
\end{split}
\end{equation*}
where the duality in the penultimate term is %justified, as 
based on the membership $Q(\cdot)\in \cQ_T$ (along with the estimate \eqref{e:regB+Q})
which yields $Q(r)B\in \cL(U,[\cD(A^\epsilon)]')$, 
combined as before with $L_s u\in C([s,T];\cD(A^\epsilon))$.
The above leads to
\begin{equation*}
\begin{split}
R_3+C_4+C_1 &=
- 2 \textrm{Re}\int_s^t \int_s^r\big(B^*e^{A^*(r-\sigma)}R^*Re^{A(r-s)}x,u(\sigma)\big)_U
\,d\sigma\,dr
\\[1mm]
& \quad +\, 2\textrm{Re} \int_s^t \int_s^r\big(B^*e^{A^*(r-\sigma)}Q(r)BB^*Q(r)e^{A(r-s)}x,
u(\sigma)\big)_U\,d\sigma dr
\\[1mm]
& \quad 
+ 2 \textrm{Re} \int_s^t (B^*Q(\sigma)e^{A(\sigma-s)}x,u(\sigma)\big)_U\,d\sigma\,, 
\end{split}
\end{equation*}
which can be rewritten, exchanging the order of integration, as follows:
\begin{equation}\label{e:seconda}
\begin{split}
R_3+C_4+C_1 =&
- 2 \textrm{Re}\int_s^t \Big(B^*\Big\{\int_\sigma^t e^{A^*(r-\sigma)}R^*R \big[e^{A(r-s)}x\big]
\,dr
\\[1mm]
& \qquad\quad -\int_\sigma^t e^{A^*(r-\sigma)}Q(r)BB^*Q(r)e^{A(r-\sigma)}\big[e^{A(\sigma-s)}x\big]\,dr
\\[1mm]
& \qquad\quad -Q(\sigma)\big[e^{A(\sigma-s)}x\big]\Big\},u(\sigma)\Big)_U\,d\sigma\,. 
\end{split}
\end{equation}

Let us focus on the expression inside the curly bracket.
Because $Q(\cdot)$ solves the IRE \eqref{e:ire}, as well as its second form \eqref{e:ire_bis} valid for any pair $x,y\in Y$, then the following identity -- a {\em strong} form of the IRE, when $Q(\cdot)$ is unknown -- holds true:
\begin{equation*}
\begin{split}
& e^{A^*(t-\sigma)}Q(t)e^{A(t-\sigma)}z= Q(\sigma)z-\int_\sigma^t e^{A^*(r-\sigma)}R^*R z\,dr
\\[1mm]
& \myspace+\int_\sigma^t e^{A^*(r-\sigma)}Q(r)BB^*Q(r)e^{A(r-\sigma)}z\,dr\,, \quad 0\le \sigma\le t\le T\,,
\; z\in Y\,.
\end{split}
\end{equation*}
Thus, returning to \eqref{e:seconda} with this information and setting in particular 
$z=e^{A(\sigma-s)}x$,
we find that $R_3+C_4+C_1$ simply reads as follows:
\begin{equation}\label{e:seconda-bis}
\begin{split}
R_3+C_4+C_1 
&= -2 \textrm{Re}\int_s^t \big(B^*e^{A^*(t-\sigma)}Q(t)e^{A(t-s)}x,u(\sigma)\big)_U\,d\sigma
\\[1mm]
&= -2 \textrm{Re}\big(Q(t)e^{A(t-s)}x,L_su(t)\big)_Y\,.
\end{split}
\end{equation}

We examine next the sum 
\begin{equation*}
R_2+C_5 =-\int_s^t \|RL_s u(r)\|_Z^2\,dr + \int_s^t \big\|B^*Q(r)L_su(r)\big\|_U^2\,dr\,,
\end{equation*}
where, again, since $u\in L^\infty(s,T;U)\subset L^{q'}(s,T;U)$, 
we know from Lemma~\ref{l:improved} that $L_su\in C([s,T];\cD(A^\epsilon))$.
Consequently, one gets
\begin{equation} \label{e:pre-quarta}
\begin{split}
& R_2+C_5 = - \int_s^t \big\langle \big[R^*R-Q(r)BB^*Q(r)\big]L_s u(r),
L_s u(r)\big\rangle_{[\cD(A^\epsilon)]',\cD(A^\epsilon)} \,dr
\\[1mm]
& \; = - \textrm{Re}\int_s^t \Big({A^*}^{-\epsilon}
\big[R^*R-Q(r)BB^*Q(r)\big]A^{-\epsilon}\,
\int_s^r A^\epsilon e^{A(r-\lambda)} Bu(\lambda)\,d\lambda,
\\[0.5mm] 
& \myspace\myspace\myspace\int_s^r A^\epsilon e^{A(r-\mu)} Bu(\mu)\,d\mu\Big)_Y\,dr\,.
\end{split}
\end{equation}
It is important to emphasize that in going from the duality to the inner product in \eqref{e:pre-quarta},
two facts have been crucially used, besides $Q(\cdot)\in \cQ_T$: 
the hypothesis \eqref{e:keyasR} on the observation operator $R$ (that is iiib) of the Assumptions~\ref{a:ipo_1}), and once again, Lemma~\ref{l:improved}. 
Further handling of the right hand side of \eqref{e:pre-quarta} leads to the triple integral
\begin{equation*} 
R_2+C_5 = - \textrm{Re}\int_s^t I(r,s) \,dr\,,
\end{equation*}
having set
\begin{equation*}
I(r,s)=\int_s^r \!\!\int_s^r 
\big(B^*e^{A^*(r-\mu)}\,\big[R^*R-Q(r)BB^*Q(r)\big] e^{A(r-\lambda)} B u(\lambda),u(\mu)\big)_U\,
d\lambda\,d\mu\,.
\end{equation*}

Let us focus on the inner double integral $I(r,s)$. 
We note that this integral pertains to a symmetric function of $(\lambda,\mu)$,
and hence the integral over the square $[s,r] \times [s,r]$ can be replaced by twice the integral
over the triangle $$\{(\lambda,\mu)\colon s \le \mu \le \lambda \le r\}\,.$$
It follows that 
\begin{equation*} 
\begin{split}
& I(r,s) = 2\,\int_s^r \!\! d\lambda \int_s^\lambda \!\! d\mu\,
\big(B^*e^{A^*(r-\mu)}\,\big[R^*R-Q(r)BB^*Q(r)\big] e^{A(r-\lambda)} B u(\lambda),u(\mu)\big)_U
\\[1mm]
& \, = 2\,\int_s^r \!\! \Big[\int_s^\lambda \!\!
\big(B^*e^{A^*(\lambda-\mu)}\,e^{A^*(r-\lambda)}\,\big[R^*R-Q(r)BB^*Q(r)\big] e^{A(r-\lambda)} B u(\lambda),u(\mu)\big)_U\, d\mu\,\Big]\,d\lambda 
\\[1mm]
& \, = 2\,\int_s^r \!\! 
\big(\,e^{A^*(r-\lambda)}\,\big[R^*R-Q(r)BB^*Q(r)\big] e^{A(r-\lambda)} B u(\lambda),\int_s^\lambda e^{A(\lambda-\mu)}Bu(\mu)\,d\mu\big)_Y\, d\lambda
\\[1mm]
& \,= 2\,\int_s^r \!\! 
\big(\,e^{A^*(r-\lambda)}\,\big[R^*R-Q(r)BB^*Q(r)\big] e^{A(r-\lambda)} B u(\lambda),L_su(\lambda)\big)_Y\,d\lambda\,.
\end{split}
\end{equation*}
Inserting the expression of $I(r,s)$ obtained above in the outer integral yields
\begin{equation*} 
\begin{split}
& R_2+C_5 
\\[1mm]
& \; = - 2\textrm{Re}\int_s^t \!\!\int_s^r 
\big(\,e^{A^*(r-\lambda)}\,\big[R^*R-Q(r)BB^*Q(r)\big] e^{A(r-\lambda)} B u(\lambda),L_su(\lambda)\big)_Y
d\lambda\,dr\,;
\end{split}
\end{equation*}
next we exchange the order of integration and also move the first argument of the inner product,
to achieve 
\begin{equation} \label{e:intermedia}
\begin{split}
& R_2+C_5
\\[1mm]
& \; = 
- 2\textrm{Re}\int_s^t \!\!\int_\lambda^t 
\big(\,e^{A^*(r-\lambda)}\,\big[R^*R-Q(r)BB^*Q(r)\big] e^{A(r-\lambda)} B u(\lambda),L_su(\lambda)\big)_Y
dr\,d\lambda
\\[1mm]
& \; = 
- 2\textrm{Re}\int_s^t \!\!\int_\lambda^t 
\big(u(\lambda),B^*e^{A^*(r-\lambda)} \big[R^*R-Q(r)BB^*Q(r)\big]e^{A(r-\lambda)} L_su(\lambda)\big)_Y
dr\,d\lambda
\\[1mm]
& \; = 
- 2\textrm{Re}\int_s^t \big(u(\lambda),
B^*\int_\lambda^t e^{A^*(r-\lambda)}\big[R^*R-Q(r)BB^*Q(r)\big]e^{A(r-\lambda)}\,
L_s u(\lambda)\,dr\big)_U\,d\lambda\,.
\end{split}
\end{equation}

It is apparent that the second form \eqref{e:ire_bis} of the IRE (with $\lambda$ in place of $s$)
-- in fact, a strong form of it --
provides once more the tool, just like in deriving \eqref{e:seconda-bis} from \eqref{e:seconda}.
With $z = L_s u(\lambda)$, replace the integral
\begin{equation*}
\int_\lambda^t e^{A^*(r-\lambda)}\big[R^*R-Q(r)BB^*Q(r)\big]e^{A(r-\lambda)}z\, dr
\end{equation*}
by $[Q(\lambda)- e^{A^*(t-\lambda)} Q(t)e^{A(t-\lambda)}]z$, to find
\begin{equation}\label{e:quinta}
\begin{split}
& R_2+C_5 = 
-2\textrm{Re}\int_s^t\big(u(\lambda),B^*\big[Q(\lambda)- e^{A^*(t-\lambda)} Q(t)e^{A(t-\lambda)}\big]L_su(\lambda)\big)_U\,d\lambda
\\[1mm]
& \qquad = - 2 \textrm{Re} \int_s^t \big(B^*\big[Q(\lambda) - e^{A^*(t-\lambda)}Q(t)e^{A(t-\lambda)}\big] L_s u(\lambda),u(\lambda)\big)_U \,d\lambda\,.
\end{split}
\end{equation}

Thus, adding $C_2$ to \eqref{e:quinta}, we see that a useful simplification occurs, as detailed below:
\begin{equation*}
\begin{split}
R_2+C_5+C_2 & = -2 \textrm{Re} \int_s^t \big(B^*Q(\lambda)L_s u(\lambda),u(\lambda)\big)_U \,d\lambda
\\[1mm]
& \qquad +2 \textrm{Re} \int_s^t \big(B^*e^{A^*(t-\lambda)}Q(t)e^{A(t-\lambda)} L_su(\lambda),u(\lambda)\big)_U \,d\lambda
\\[1mm]
& \qquad
+2 \textrm{Re}\int_s^t \big(B^*Q(r)L_s u(r),u(r)\big)_U\,dr
\\[1mm]
& =2 \textrm{Re} \int_s^t\!\!\int_s^\lambda \big(Q(t)e^{A(t-\sigma)} Bu(\sigma),e^{A(t-\lambda)} Bu(\lambda)\big)_Y \,d\sigma\,d\lambda\,.
\end{split}
\end{equation*}
Owing to the simmetry of the latter integrand in $(\sigma,\lambda)$, we may replace twice the integral over
the triangle $\{(\lambda,\sigma)\colon s \le \lambda\le \sigma \le t\}$ by the integral over the square 
$[s,t] \times [s,t]$, and finally get
\begin{equation}\label{e:sesta}
\begin{split}
R_2+C_5+C_2 
& = \textrm{Re} \int_s^t\!\!\int_s^t \big(Q(t)e^{A(t-\sigma)} Bu(\sigma),e^{A(t-\lambda)} Bu(\lambda)\big)_Y \,d\lambda\,d\sigma
\\[1mm]
& = \textrm{Re} \big(Q(t)L_su(t),L_su(t)\big)_Y=\big(Q(t)L_su(t),L_su(t)\big)_Y\,.
\end{split}
\end{equation}

Combining \eqref{e:sesta} with \eqref{e:prima} and \eqref{e:seconda-bis}, we finally obtain
\begin{equation*}
\begin{split}
\sum_{i=1}^3 R_i+ \sum_{j=1}^5 C_j & 
= \big(Q(t)e^{A(t-s)}x,e^{A(t-s)}x\big)_Y - (Q(s)x,x)_Y
\\[-1mm]
&\qquad + 2 \textrm{Re} (Q(t)e^{A(t-s)}x,L_su(t))_Y + (Q(t)L_su(t),L_su(t))_Y
\\[2mm]
& = (Q(t)y(t),y(t))_Y - (Q(s)x,x)_Y\,,
\end{split}
\end{equation*}
which establishes the fundamental identity \eqref{e:fundamental} in the case $u\in L^\infty(s,T;U)$.
Finally, the identity extends to $u\in L^2(s,T;U)$ by density, which concludes the proof of 
Lemma~\ref{l:id-fondamentale}.

\end{proof}

% CLOSED LOOP EQ.
 
We next introduce an integral equation that involves a given operator solution $Q(t)$ to the Riccati equation corresponding to optimal control problem \eqref{e:state-eq}--\eqref{e:cost}.
Once uniqueness for the DRE \eqref{e:DRE} is established, so that $Q(t)$ must coincide with the Riccati operator $P(t)$, then it will be clear that the said integral equation (\eqref{e:closedloop} below)
is nothing but the well known {\em closed-loop equation}, of central importance for the synthesis of the optimal control.
(This justifies the use of the term ``closed-loop equation'' for \eqref{e:closedloop}).

As we shall see, the following Lemma~\ref{l:closedloop} and (the independent) Lemma~\ref{l:id-fondamentale}
constitute the core elements for the proof of Theorem~\ref{t:dre-uniqueness}. 

\begin{lemma}\label{l:closedloop} 
Let $\epsilon$ be as in iii) of Assumptions~\ref{a:ipo_1}. 
Let $Q\in \cQ_T$, where $\cQ_T$ is the class defined by \eqref{e:class-tfinite}. 
Then, for every $x\in \cD(A^\epsilon)$, the closed loop equation
\begin{equation}\label{e:closedloop}
y(t) = e^{At}x - \int_0^t e^{A(t-s)}BB^*Q(\sigma) y(\sigma)\,d\sigma\,, \qquad t\in [0,T]\,,
\end{equation}
has a unique solution in the space
\begin{equation} \label{e:spcllo}
X=\Big\{y\in C([0,T];\cD(A^\epsilon))\colon 
\sup_{t\in [0,T]} \big(e^{-rt} \|y(t)\|_{\cD(A^\varepsilon)}) < \infty\Big\}
\end{equation}
endowed with the norm
\begin{equation*}
\|y\|_{X,r} = \sup_{t\in [0,T]} e^{-rt} \|y(t)\|_{\cD(A^\epsilon)}\,, \qquad y\in X\,,
\end{equation*}
provided $r>0$ is chosen sufficiently large.
\end{lemma}

% Dim.

\begin{proof}
With $x\in \cD(A^\epsilon)$, we set $E(t)=e^{At}x$.
By semigroup theory we know that $E(\cdot)\in C([0,T];\cD(A^\epsilon))$; even more,
since $e^{At}$ is exponentially stable, it holds $E(\cdot)\in X$ provided $r$ is sufficiently large.
As the integral equation \eqref{e:closedloop} has the clear form 
\begin{equation*} 
y(t)+\big[LB^*Q(\cdot)y(\cdot)\big](t) = E(t)\,, \quad t\in [0,T]\,,
\end{equation*}
we appeal to a classical argument of functional analysis: we will prove that $LB^*Q(\cdot)$ is a contraction mapping in $X$, having chosen $r$ sufficiently large. 
This will in turn imply that $I+LB^*Q(\cdot)$ is invertible in $X$, thus providing the sought unique solution
to \eqref{e:closedloop}.

For each $y\in X$, $z\in \cD({A^*}^\epsilon)$, $t\in[0,T]$, we have by Lemma~\ref{l:gain-on-functions}
\begin{align*}
& \big|(e^{-rt} LB^*Q(\cdot)y(\cdot),{A^*}^\epsilon z)_Y\big| 
\\[1mm]
& \quad =\Big|\int_0^t e^{-r(t-s)}\,\big(B^*Q(s)e^{-rs}y(s),B^*e^{A^*(t-s)}{A^*}^\epsilon z\big)_U\,ds\Big|
\\[1mm] 
& \quad \le \int_0^t e^{-r(t-s)}\|B^*Q(\cdot)e^{-r\cdot}y(\cdot)\|_{C([0,T];U)}\, 
\big\|B^*e^{(t-s)A^*}{A^*}^\epsilon z\big\|_U\,ds 
\\[1mm]
& \quad \le \|B^*Q(\cdot)\|_{\cL(C([0,T];\cD(A^\epsilon)),C([0,T];U))}\,\|y\|_{X,r} \int_0^t e^{-r\sigma}
\|B^*e^{A^*\sigma}{A^*}^\epsilon z\|_U\,d\sigma
\\[1mm]
& \quad \le \|B^*Q(\cdot)\|_{\cL(C([0,T];\cD(A^\epsilon)),C([0,T];U))}\,\|y\|_{X,r}\, 
\Big[\int_0^t e^{-r\sigma q'}\,d\sigma\Big]^{1/q'}
\|B^*e^{\cdot A^*}{A^*}^\epsilon z\|_{L^q(0,T;U)} 
\\
& \quad \le \|B^*Q(\cdot)\|_{\cL(C([0,T];\cD(A^\epsilon)),C([0,T];U))}\,\frac{1}{(rq')^{1/q'}}\,
\|B^*e^{A^*\cdot }{A^*}^\epsilon\|_{\cL(Y,L^q(0,T;U))}\,
\|z\|_Y\,\|y\|_{X,r}\,. 
\end{align*}
We note that in going from the antepenultimate to the penultimate estimate we used iiic) of the Assumptions~\ref{a:ipo_1}.
Therefore, there exist positive constants $c, c'$ such that
\begin{equation*} 
\big\|e^{-rt}\big[L\,B^*Q(\cdot)y(\cdot)\big](t)\big\|_{\cD(A^\epsilon)}
\le \frac{c}{(rq')^{1/q'}}\,\|y\|_{X,r}
\le \frac{c'}{r^{1/q'}}\,\|y\|_{X,r}\,,
\end{equation*}
so that by taking a sufficiently large $r$ we see that $LB^*Q(\cdot)$ is a contraction mapping 
in $X$.
The conclusion of the Lemma follows. 
\end{proof}

\smallskip
Uniqueness for the DRE is now a consequence of Lemmas~\ref{l:id-fondamentale} and \ref{l:closedloop},
its proof following a somewhat familiar path.  

% Proof of uniqueness T finite

\smallskip
\noindent
{\em Proof of Theorem~\ref{t:dre-uniqueness}.} 
For the optimal pair $(\hat{y},\hat{u})$ corresponding to the initial state $x\in Y$
it holds
\begin{equation*}
(P(s)x,x)_Y=J(\hat{u})=\int_s^T \Big(\|R\hat{y}(r)\|_Z^2+\|\hat{u}(r)\|_U^2\Big)\,dr\,, \quad 0\le s\le T\,,
\end{equation*}
where $P(\cdot)$ is the Riccati operator defined in \eqref{opRic}, i.e.
\begin{equation*}
P(t)x = \int_t^T e^{A^*(r-t)} R^*R \Phi(r,t)x\, dr\,, \qquad x\in Y\,,
\end{equation*} 
while $\Phi(r,t)$ denotes the evolution operator
\begin{equation*}
\Phi(r,t)x = e^{A(r-t)}x+L_t\hat{u}(r)\,, \qquad r\in [t,T]\,.
\end{equation*}

Let $Q(\cdot)\in \cQ_T$ be another solution to the DRE \eqref{e:DRE}: by Lemma~\ref{l:DtoI} $Q(\cdot)$
solves the IRE \eqref{e:ire} as well; then, with $u\in L^2(s,T;U)$ and $x\in \cD(A^\epsilon)$,
the identity \eqref{e:fundamental} holds true by Lemma~\ref{l:id-fondamentale}.
With $t=T$, since $Q(T)=0$, from \eqref{e:fundamental} we see that
\begin{equation*}
\begin{split}
(Q(s)x,x)_Y &= \int_s^T \left[\|Ry(r)\|_Z^2 + \|u(r)\|_U^2\right]\,dr - \int_s^T \|u(r)+B^*Q(r)y(r)\|_U^2 \,dr 
\\[1mm]
& \le \int_s^T \left[\|Ry(r)\|_Z^2 + \|u(r)\|_U^2\right]\,dr=J(u)\,.
\end{split}
\end{equation*}
In particular, when $u=\hat{u}$, we establish
\begin{equation} \label{e:for-polarization_1}
(Q(s)x,x)_Y\le J(\hat{u})=(P(s)x,x)_Y\qquad 
\forall s\in [0,T]\,, \; \forall x\in \cD(A^\epsilon)\,.
\end{equation}

Conversely, let $y(\cdot)$ be the solution to the closed-loop equation \eqref{e:closedloop} corresponding to $x\in \cD(A^\epsilon)$, guaranteed by Lemma~\ref{l:closedloop}, and 
let $u(\cdot)=-B^*Q(\cdot)y(\cdot)$.
By construction $u\in L^2(s,T;U)$, and the fundamental identity becomes 
\begin{equation*}
(Q(s)x,x)_Y = \int_s^t \Big(\|Ry(r)\|_Z^2 + \|u(r)\|_U^2\Big)\,dr+(Q(t)y(t),y(t))_Y\,,
\end{equation*}
which in turn gives, for $t=T$,
\begin{equation}\label{e:for-polarization_2}
(Q(s)x,x)_Y = J(u)\ge J(\hat{u})= (P(s)x,x)_Y \qquad \forall s\in [0,T]\,, \;\forall x\in \cD(A^\epsilon)\,.
\end{equation}
The inequality \eqref{e:for-polarization_2}, combined with \eqref{e:for-polarization_1}, establishes
-- via the usual polarization (first) and density (next) arguments -- $Q(s)\equiv P(s)$ on $[0,T]$, 
as desired.

\qed

% T = +\infty

\subsection{Infinite time interval, algebraic Riccati equations}
In this Section we prove our second main result, that is Theorem~\ref{t:are-uniqueness}, which pertains
to uniqueness for the algebraic Riccati equation \eqref{e:ARE}, under the standing Assumptions~\ref{a:ipo_2}.
Instrumental results are the counterparts of Lemmas~\ref{l:id-fondamentale} and \ref{l:closedloop}, along with the integral form \eqref{e:are-integral} of the ARE, already obtained in Section~\ref{s:uniqueness_1};
see Lemma~\ref{l:are-integral} therein. 

% Preliminary Lemmas

The first Lemma is the infinite time horizon version of the fundamental identity established
in Lemma~\ref{l:id-fondamentale}. 

% FUNDAMENTAL IDENTITY (T=+\infty)

\begin{lemma}[Fundamental identity ($T=+\infty$)]\label{l:id_fond_infty}
Recall the class $\cQ$ defined in \eqref{e:classeQ}.
Let $Q\in \cQ$ be a solution to the integral Riccati equation \eqref{e:are-integral}.
With $u\in L^2_{\textrm{loc}}(0,\infty;U)$ and $x \in \cD(A^\epsilon)$, let $y(\cdot)$ 
be the semigroup solution to the state equation \eqref{e:state-eq} corresponding to $u(\cdot)$,
with initial state $x$, given by \eqref{e:mild}.
Then, the following identity holds true, for any $t\ge 0$:
\begin{equation} \label{idbase} 
\begin{split}
& (Qy(t),y(t))_Y - (Qx,x)_Y  
\\[2mm]
& \qquad =-\int_0^t \Big(\|Ry(s)\|_Z^2 + \|u(s)\|_U^2\Big)\,ds 
+ \int_0^t \big\|u(s)+B^*Qy(s)\big\|_U^2 \,ds\,.
\end{split}
\end{equation}

\end{lemma} 

\begin{proof} 
It suffices to proceed along the lines of the proof of Lemma~\ref{l:id-fondamentale}, replacing the interval
$[s,t]$ by $[0,t]$ and assuming initially $u\in L^\infty_{\textrm{loc}}(0,\infty;U)$; the proof is 
actually slightly simpler, since here $Q$ is independent of $t$. 
The details are omitted for the sake of conciseness.
\end{proof}

The next Lemma is the infinite time horizon version of Lemma~\ref{l:closedloop}, dealing with
an integral equation which -- once uniquenss for the ARE is ascertained -- will turn out to be
the closed-loop equation.

% CLOSED LOOP EQ. T=+\infty

\begin{lemma}\label{l:closedloopinfty}
Let $\epsilon$ be as in the Assumptions~\ref{a:ipo_2}.
Recall the class $\cQ$ defined by \eqref{e:classeQ}, and let $Q\in \cQ$.
For every $x\in \cD(A^\epsilon)$ and a suitably large $r>0$ there exists a unique solution
$y(\cdot)$ to the closed loop equation
\begin{equation}  \label{e:closed-loop-halfline}
y(t) = e^{At}x - \int_0^t e^{A(t-s)}BB^*Q y(s)\,ds\,, \qquad t>0\,,
\end{equation}
in the space 
\begin{equation} \label{spcllo}
X=\big\{y\in C([0,\infty);\cD(A^\epsilon))\colon \quad 
\sup_{t\ge 0} e^{-rt} \|y(t)\|_{\cD(A^\epsilon)} < \infty\big\}
\end{equation}
endowed with the norm 
\begin{equation*}
\|y\|_{X,r}=\sup_{t>0} e^{-rt}\|y(t)\|_{D(A^\epsilon)} \quad \forall y\in X\,, \quad r>0\,.
\end{equation*}

\end{lemma}

\begin{proof}
The argument is pretty much the same employed in the proof of Lemma~\ref{l:closedloop}.
A technically decisive (distinct) element here comes from the extended (and enhanced) regularity in time of the operator $B^*e^{A^*\cdot} {A^*}^\epsilon$ over the half line $[0,\infty)$, which is guaranteed by
\cite[Proposition~3.2]{abl_2013}, recalled here as Proposition~\ref{p:2013-P3.2}.
The computation is included for the reader's convenience.

Let $x\in \cD(A^\epsilon)$ be given.
By setting $E(t)=e^{At}x$, and recalling the input-to-state map $L$, the integral equation 
\eqref{e:closed-loop-halfline} reads as $\big([I+LB^*Q]y(\cdot)\big)(t)=E(t)$, in short.
For any function $y(\cdot)\in X$ and any $z\in \cD({A^*}^\epsilon)$, we have
\begin{equation*}
\begin{split}
& |(e^{-rt}LB^*Qy(t),A^{*\epsilon}z)_Y
= \left|\int_0^t e^{-r(t-s)}(B^*Qy(s) e^{-rs}, B^* e^{A^*(t-s)}{A^*}^\epsilon z)_Y\,ds\right|
\\[1mm]
& \qquad \le \int_0^t e^{-r(t-s)}\|B^*Q\|_ {\cL(\cD(A^\epsilon),U)}\|y\|_{X,r}\, 
e^{-\delta (t-s)}\,\|e^{\delta (t-s)}B^*e^{A^*(t-s)}{A^*}^\epsilon z\|_U\,ds
\\[1mm]
& \qquad \le \|B^*Q\|_ {\cL(\cD(A^\epsilon),U)}\|y\|_{X,r}\,
\Big(\int_0^t e^{-(r+\delta)(t-s)q'}\,ds\Big)^{1/q'}\,\cdot
\\[1mm]
& \myspace\myspace\qquad
\cdot\Big(\int_0^t \|e^{\delta (t-s)}B^*e^{A^*(t-s)}{A^*}^\epsilon z\|_U^q\,ds \Big)^{1/q}
\\[1mm]
& \qquad \le 
\frac1{[(r+\delta)q']^{1/q'}}\,\|B^*Q\|_{\cL(\cD(A^\epsilon),U)}\,
\|e^{\delta\cdot}B^*e^{A^*\cdot} {A^*}^\epsilon\|_{\cL(Y,L^q(0,\infty;U))}\,\|y\|_{X,r}\,\|z\|_Y\,,
\end{split}
\end{equation*}
where $\delta$ belongs to the interval $(0,\omega\wedge \eta)$ ($\omega$ and $\eta$ being like in the Assumptions~\ref{a:ipo_2}).
The above estimate implies readily that there exists a constant $C>0$ such that
\begin{equation*}
\|LB^*Qy\|_{X,r} \le \frac{C}{(r+\delta)^{1/q'}}\, \|y\|_{X,r}
\|e^{\delta\cdot}B^*e^{A^*\cdot}A^{*\epsilon}\|_{\cL(Y,L^q(0,\infty;U))}
\end{equation*}
so that 
\begin{equation*}
\|LB^*Qy\|_{X,r} \le \frac12 \|y\|_{X,r}\,,
\end{equation*}
provided $r$ is sufficiently large.
The conclusive argument is standard.
\end{proof}

{\em Proof of Theorem~\ref{t:are-uniqueness}.} 
Let $y_0\in Y$, and let $(\hat{y},\hat{u})$ the optimal pair of the optimal control problem
\eqref{e:state-eq}-\eqref{e:cost} (with $T=+\infty$), corresponding to the initial state $y_0$.
Recall that
\begin{equation*}
(Py_0,y_0)_Y = J(\hat{u}) = \int_0^\infty \|R\hat{y}(s)\|^2_Z \, ds 
+ \int_0^\infty \|\hat{u}(s)\|_U^2 \,ds\,,
\end{equation*}
where the (optimal cost) operator $P$ is defined in terms of the optimal state semigroup 
$\Phi(t)$ via \eqref{e:optimal-cost-op}.
In addition, $P$ belongs to the class $\cQ$ and solves the ARE \eqref{e:ARE}; consequently,
by Lemma~\ref{l:are-integral} $P$ solves the integral form \eqref{e:are-integral} of the ARE.

Let now $Q\in \cQ$ be another solution to the ARE. 
By Lemma~\ref{l:id_fond_infty}, we know that for any given $y_0\in \cD(A^\epsilon)$, 
and any admissible control $u(\cdot)$, the identity \eqref{idbase} holds true (with $x$ replaced by
$y_0$), where $y(\cdot)$ is the solution to the state equation corresponding to the control $u$ and 
the initial state $y_0$.
Consequently,
\begin{equation*}
(Qy_0,y_0)_Y \le (Qy(t),y(t))_Y + J(u) \qquad 
\forall u\in L^2_{\textrm{loc}}(0,\infty;U)\,, \;\forall t>0\,;
\end{equation*}
by choosing in particular the admissible pair $(y_T,u_T)$ defined as follows,
\begin{equation*}
u_T = \hat{u}\cdot \chi_{[0,T]}\,, \qquad 
y_T(t) = 
\begin{cases}
\hat{y}(t) & \textrm{if $t\le T$} 
\\[2mm]
e^{At}y_0 + e^{A(t-T)}L\hat{u}(T) & \textrm{if $t>T$,}
\end{cases}
\end{equation*}
we find $(Qy_0,y_0)_Y \le (Qy_T(t),y_T(t))_Y + J(u_T)$, valid for arbitrary $t\ge T>0$.
By letting $t\to +\infty$ in the previous inequality, one obtains readily
\begin{equation} \label{e:prepre-polarization_1}
(Qy_0,y_0)_Y \le J(u_T) \qquad \forall y_0\in \cD(A^\epsilon)\, \quad \forall T>0\,,
\end{equation}
in view of the fact that the semigroup $e^{At}$ decays exponentially; so
$\|y_T(t)\|_Y\longrightarrow 0$, as $t\to +\infty$.

Observe now that
\begin{equation*}
\begin{split}
J(u_T) &= \int_0^\infty \|Ry_T(s)\|^2_Z \, ds + \int_0^\infty \|u_T(s)\|_U^2 \,ds
\\[2mm]
& = \int_0^T \|R\hat{y}(s)\|^2_Z \, ds\,
+ \int_T^\infty \big\|R\big(e^{sA}y_0+e^{(s-T)A}L\hat{u}(T)\big)\big\|^2\,ds
+ \int_0^T \|\hat{u}(s)\|_U^2 \,ds\,, 
\end{split}
\end{equation*}
so that $J(u_T)\longrightarrow J(\hat{u})$, as $T\rightarrow +\infty$.
Keeping this in mind, return to \eqref{e:prepre-polarization_1} and let $T\to +\infty$ to find
\begin{equation} \label{e:pre-polarization_1}
(Qy_0,y_0)_Y \le J(\hat{u})=(Py_0,y_0)_Y \qquad \forall y_0\in \cD(A^\epsilon)\,.
\end{equation}

\smallskip
On the other hand, given $y_0\in \cD(A^\epsilon)$ (and still with $Q\in \cQ$ another solution
to the ARE), let $y(\cdot)$ be the solution to the closed loop equation guaranteed 
by Lemma~\ref{l:closedloopinfty}; by construction $y\in L^2_{\textrm{loc}}(0,\infty;Y)$. 
Take now the control $u(\cdot)=-B^*Qy(\cdot)$, which belongs to $L^2_{\textrm{loc}}(0,\infty;U)$.
Then, the identity \eqref{idbase} holds true for any positive $t$, that is
\begin{equation}
\begin{split}
& (Qy_0,y_0)_Y = (Qy(t),y(t))_Y + \int_0^t \Big(\|Ry(s)\|_Z^2 + \|u(s)\|_U^2\Big)\,ds
\\[2mm]
& \qquad - \cancel{\int_0^t \big\|u(s)+B^*Qy(s)\big\|_U^2 \,ds}
\ge \int_0^t \|Ry(s)\|_Z^2\,ds + \int_0^t \|u(s)\|_U^2\,ds\,.
\end{split}
\end{equation}

As $t\to +\infty$, this shows that $Ry\in L^2(0,\infty;Z)$, $u\in L^2(0,\infty;U)$,
as well as $(Qy_0,y_0)_Y \ge J(u)$.
By minimality we find
\begin{equation} \label{e:pre-polarization_2}
(Qy_0,y_0)_Y \ge J(u)\ge J(\hat{u}) = (Py_0,y_0)_Y \qquad \forall y_0\in \cD(A^\epsilon)\,.
\end{equation}
On the basis of \eqref{e:pre-polarization_1} and \eqref{e:pre-polarization_2}, a standard polarization (first)
and density (next) argument confirms that $Q=P$, thereby concluding the proof of Theorem \ref{t:are-uniqueness}.
\qed

% APPENDICE

\appendix

\section{Instrumental results} \label{a:appendix}
In this appendix we gather several results (some old, some new) which single out certain regularity properties
-- in time and space -- of
\begin{itemize}
\item
the input-to-state map $L$,
\item
the operator $B^*Q(\cdot)$, when $Q(t)\in \cQ_T$, 
\item
the operator $B^*e^{A^*t}{A^*}^\epsilon$ and its adjoint.
\end{itemize}
All of them stem from the Assumptions~\ref{a:ipo_1} or \ref{a:ipo_2} on the (dynamics and control) 
operators $A$ and $B$.
The role played by the assertions of the novel Lemma~\ref{l:improved} and Lemma~\ref{l:gain-on-functions}
in the proofs of our uniqueness results is absolutely critical.

\smallskip
Initially, it is useful to recall from \cite{abl_2005} and \cite{abl_2013} the basic regularity 
properties of the input-to-state map $L$.
The first result pertains to the finite time horizon problem.
The reader is referred to \cite[Appendix~B]{abl_2005} for the details of the computations 
leading to the various statements in the following Proposition.

% OLD: Lemma di regolarita' #1

\begin{proposition}[\cite{abl_2005}, Proposition~B.3] \label{p:2005-B3} 
Let $L_s$ be the operator defined by
\begin{equation} \label{e:convolution-in-st}
L_s\colon u(\cdot) \longrightarrow 
(L_su)(t): = \int_s^t e^{A(t-r}Bu(r)\,dr\,, \qquad 0\le s\le t\le T\,.
\end{equation}
Under the Assumptions~\ref{a:ipo_1}, 
%on the (dynamics, control, and observation) operators $A$, $B$, $R$, 
the following regularity results hold true.

\begin{enumerate}

\item
If $p=1$, then $L_s \in \cL(L^1(s,T;U),L^{1/\gamma}(s,T;[\cD({A^*}^\epsilon)]')$;

\item 
if $1 < p < \frac{1}{1-\gamma}$, then 
$L_s \in \cL(L^p(s,T;U),L^r(s,T;Y))$, with $r =\frac{p}{1-(1-\gamma)p}$;

\item 
if $p=\frac{1}{1-\gamma}$, then $L_s \in \cL(L^p(s,T;U),L^r(s,T;Y))$ for all $r\in [1,\infty)$;

\item 
if $p>\frac{1}{1-\gamma}$, then $L_s \in \cL(L^p(s,T;U),C([s,T];Y))$.

\end{enumerate}
Moreover, in all cases the norm of $L_s$ does not depend on $s$.
\end{proposition}

The space regularity in the last assertion can be actually enhanced. 
To be more precise, $L_s$ maps control functions $u(\cdot)$ which belong to $L^{q'}(s,T;U)$ into 
functions which take values in $\cD(A^\epsilon)$ ($q'$ being the conjugate exponent of $q$ in the Assumptions~\ref{a:ipo_1}).
We highlight this property -- appparently left out of the work \cite{abl_2005} -- as a separate result, since it will be used throughout in the paper.
The proof is omitted, as it is akin to (and somewhat simpler than) the one carried out to establish 
assertion (v) of the subsequent Proposition~\ref{p:2013-P3.6}.

% LEMMA NUOVO (ex Lemma di regolarita' 2)

\begin{lemma} \label{l:improved}
Let $\epsilon$ and $q$ be as in (iii) of the Assumptions~\ref{a:ipo_1}. 
Then, for the operator $L_s$ defined in \eqref{e:convolution-in-st} we have 
\begin{equation*}
L_s \in \cL(L^{q'}(s,T;U),C([s,T];\cD(A^\epsilon)))\,.
\end{equation*}
\end{lemma}

\smallskip
A counterpart of Proposition~\ref{p:2005-B3} specific for the infinite time horizon problem 
was proved in \cite[Proposition~3.6]{abl_2013}.
The collection of findings on the regularity of the input-to-state map $L$ is recorded here for the reader's convenience.

\begin{proposition}[\cite{abl_2013}, Proposition~3.6] \label{p:2013-P3.6} 
Let $L$ be the operator defined by
\begin{equation*}
L\colon u(\cdot) \longrightarrow 
(L u)(t): = \int_0^t e^{A(t-r}Bu(r)\,dr\,, \qquad t\ge 0\,.
\end{equation*}
Under the Assumptions~\ref{a:ipo_2}, 
%on the (dynamics, control, and observation) operators $A$, $B$, $R$, 
the following regularity results hold true.

\begin{enumerate}

\item[(i)]
$L \in \cL(L^1(0,\infty;U),L^r(0,\infty;[\cD({A^*}^\epsilon)]')$, for any $r\in [1,1/\gamma)$;

\item[(ii)]
$L \in \cL(L^p(0,\infty;U),L^r(0,\infty;Y))$, for any $p\in (1,1/(1-\gamma))$ and any 
$r\in [p,p/(1-(1-\gamma)p)]$; 

\item [(iii)]
$L \in \cL(L^\frac{1}{1-\gamma}(0,\infty;U),L^r(0,\infty;Y))$, for any 
$r\in [1/(1-\gamma),\infty)$; 

\item [(iv)]
$L \in \cL(L^p(0,\infty;U),L^r(0,\infty;Y)\cap C_b([0,\infty);Y))$, for any 
$p\in (1/(1-\gamma),\infty)$ and any $r\in [p,\infty)$; 

\item [(v)]
$L \in \cL(L^r(0,\infty;U),C_b([0,\infty);\cD(A^\epsilon))$, 
for any $r\in [q',\infty]$.
\end{enumerate}

\end{proposition}

\smallskip
Because they occur in the present work, besides being central to the analysis of \cite{abl_2013},
we need to recall the $L^p$-spaces with weights.
Set
\begin{equation*}
L^p_g (0,\infty;X) := \big\{f\colon (0,\infty) \longrightarrow X\,, \; 
g(\cdot) f(\cdot) \in L^p(0,\infty;X)\big\}\,,
\end{equation*}
where $g\colon (0,\infty) \longrightarrow \mathbb{R}$ is a given ({\em weight}) function.
We will use more specifically the exponential weights $g(t) = e^{\delta t}$, along with the following (simplified) notation:
\begin{equation*}
L^p_\delta (0,\infty;X) := \big\{f\colon (0,\infty) \longrightarrow X\,, \; 
e^{\delta\cdot}f(\cdot) \in L^p(0,\infty;X)\big\}\,.
\end{equation*}

\begin{remark}
\begin{rm}
As pointed out in \cite[Remark~3.8]{abl_2013}, all the regularity results provided by the statements contained in the Propositions~\ref{p:2005-B3} and \ref{p:2013-P3.6} extend readily to natural analogues involving $L^p_\delta$ spaces (rather than $L^p$ ones), maintaining the respective summability exponents $p$. 
\end{rm}
\end{remark}

\medskip
We now move on to a result which clarifies the regularity of the operator $B^*Q(\cdot)$, $Q(t)\in \cQ_T$,
when acting upon {\em functions} (with values in $\cD(A^\epsilon)$) rather than on {\em vectors} 
-- namely, on elements of the space $\cD(A^\epsilon)$.

\begin{lemma}\label{l:gain-on-functions}
Let $\epsilon$ be as in (iii) of the Assumptions~\ref{a:ipo_1}. 
If $Q(\cdot)\in \cQ_T$ and $f\in C([0,T];\cD(A^\epsilon))$, then 
\begin{equation*}
B^*Q(\cdot)f(\cdot)\in C([0,T];U)\,.
\end{equation*}

\end{lemma}

\begin{proof}
We proceed along the lines of the proof of \cite[Lemma~A.3]{abl_2005}. 
Let $Q(\cdot)\in \cQ_T$ and let $t_0\in [0,T]$. 
By the definition of $\cQ_T$, there exists $c_1>0$ such that 
\begin{equation}\label{e:regB+Q}
\|B^*Q(t)z\|_U\le c_1 \|z\|_{\cD(A^\epsilon)}\qquad \forall t\in [0,T]\,, \; \forall z\in \cD(A^\epsilon)\,.
\end{equation} 
Since $f(t_0)\in \cD(A^\epsilon)$, then $B^*Q(\cdot)f(t_0)\in C([0,T];U)$.
Then
\begin{equation*}
\begin{split}
& \|B^*Q(t)f(t)- B^*Q(t_0)f(t_0)\|_U
\\[1mm]
& \qquad \le \|B^*Q(t)[f(t)-f(t_0)]\|_U + \|B^*Q(t)f(t_0)- B^*Q(t_0)f(t_0)\|_U
\\[1mm]
& \qquad \le c_1 \|f(t)-f(t_0)\|_{\cD(A^\epsilon)}+  \|[B^*Q(t)- B^*Q(t_0)]f(t_0)\|_U
=o(1)\,, \quad t\longrightarrow t_0\,.
\end{split}
\end{equation*} 

\end{proof}

\smallskip
Essential as well in this work, and more specifically in the proof of Theorem~\ref{t:are-uniqueness}, is
a stronger property of the operator $B^* e^{A^*\cdot}{A^*}^{\epsilon}$, namely, \eqref{e:da-iiic} below,
which holds true for appropriate $\delta$, under the Assumptions~\ref{a:ipo_2}.
Originally devised in \cite{abl_2013}, this result reveals that once the validity of iiic) of Assumptions~\ref{a:ipo_1} is ascertained on some bounded interval $[0,T]$, then the very same regularity estimate extends to the half line, along with an enhanced summability 
of the function $B^* e^{A^*\cdot}{A^*}^{\epsilon}x$, $x\in Y$.
The key to this is the exponential stability of the semigroup, i.e.~\eqref{e:stability};
see \cite[Proposition~3.2]{abl_2013}.

\begin{proposition}[\cite{abl_2013},~Proposition~3.2] \label{p:2013-P3.2}
Let $\omega$, $\eta$ and $\epsilon$ like in the Assumptions~\ref{a:ipo_2}.
For each $\delta \in (0,\omega\wedge \eta)$ the map 
\begin{equation*}
t \longrightarrow e^{\delta t} B^* e^{A^*t}{A^*}^{\epsilon}
\end{equation*}
has an extension which belongs to $\cL(Y,L^q(0,\infty;U))$.
In short, 
\begin{equation} \label{e:da-iiic}
B^* e^{A^*\cdot}{A^*}^{\epsilon}\in \cL(Y,L^q_\delta(0,\infty;U))\,.
\end{equation}

\end{proposition}

\medskip
We conclude providing a result that takes a more in-depth glance at the regularity of the operator
$B^* e^{A^*t}{A^*}^{\epsilon}$ and its adjoint.

\begin{lemma} \label{l:duali}
Under the Assumptions~\ref{a:ipo_2}, the following regularity results are valid,
for any $\delta\in (0,\omega\wedge \eta)$:
\begin{equation} \label{e:duali}
\begin{split}
& a) \quad e^{\delta \cdot}A^\epsilon e^{A\cdot} B \in \cL(L^{q'}(0,\infty;U),Y)\,,
\\[1mm] 
& b) \quad e^{\delta \cdot}B^* e^{A^*\cdot}{A^*}^{-\epsilon} \in \cL(L^r(0,\infty;Y),U) 
\qquad \forall r>\frac1{1-\gamma}\,. 
\end{split}
\end{equation}
The respective actions of the operators in \eqref{e:duali} are made explicit % clarified
by \eqref{e:s-star} and \eqref{e:t-star}. %, respectively.
\end{lemma}

\begin{proof}
The regularity results in \eqref{e:duali} are, in essence, dual properties of the regularity result
in Proposition~\ref{p:2013-P3.2} and of assertion A6.~in Theorem~\ref{t:theory-tinfinite}, respectively.
To infer (a), we introduce the notation $S$ for the mapping from $Y$ into $L^q(0,\infty;U)$
defined by 
\begin{equation*}
Y\ni z\longrightarrow [Sz](t) := e^{\delta t} B^*e^{A^*t}{A^*}^\epsilon z\,, \qquad\;t>0\,.
\end{equation*}
For any $z\in Y$ and any $h\in L^{q'}(0,\infty;U)$, it must be $S^*\in \cL(L^{q'}(0,\infty;U),Y)$ 
and more precisely, %along with
\begin{equation*} 
\begin{split}
\langle S^*h,z\rangle_Y 
&=\langle h,Sz\rangle_{L^{q'}(0,\infty;U),L^q(0,\infty;U)}
= \int_0^\infty \big\langle h(t),e^{\delta t}B^* e^{A^*t}{A^*}^\epsilon z\big\rangle_U\,dt
\\[1mm]
& 
= \Big\langle \int_0^\infty e^{\delta t}A^\epsilon e^{At} Bh(t)\,dt,z \Big\rangle_Y\,.
\end{split}
\end{equation*}
Therefore,
\begin{equation} \label{e:s-star}
S^*h = \int_0^\infty e^{\delta t}A^\epsilon e^{At} Bh(t)\,dt\,, \qquad h\in L^{q'}(0,\infty;Y)\,.
\end{equation} 

To achieve (b) of \eqref{e:duali}, we recall instead the assertion A6.~in Theorem~\ref{t:theory-tinfinite},
which further tells us that
\begin{equation*}
e^{\delta \cdot} A^{-\epsilon} e^{A\cdot} B\in \cL(U,L^{p}(0,\infty;Y))\,, \quad 
\text{for any $p$ such that $1\le p <\frac1{\gamma}$.}
\end{equation*}
Similarly as above, we introduce the notation $T$ for the mapping from $U$ into $L^p(0,\infty;Y)$ defined by
\begin{equation*} 
U\ni w\longrightarrow [Tw](t) := e^{\delta t}A^{-\epsilon} e^{At} Bw\,, \qquad t>0\,;
\end{equation*}
by construction, $T^*\in \cL(L^{p'}(0,\infty;Y),U)$ for all $p'>1/(1-\gamma)$.
More precisely, for any $w\in U$ and any $g\in L^{p'}(0,\infty;Y)$ we have 
\begin{equation*}
\begin{split}
\langle T^*g,w\rangle_U & = \langle g,Tw\rangle_{L^{p'}(0,\infty;Y),L^p(0,\infty;Y)} 
=\int_0^\infty \big\langle g(t),e^{\delta t} A^{-\epsilon} e^{At} Bw\big\rangle_Y\,dt
\\[1mm]
&= \Big\langle \int_0^\infty e^{\delta t}B^* e^{A^*t}{A^*}^{-\epsilon} g(t)\,dt, w\Big\rangle_U\,,
\end{split}
\end{equation*}
which establishes
\begin{equation}  \label{e:t-star}
T^*g=\int_0^\infty e^{\delta t}B^* e^{A^*t}{A^*}^{-\epsilon} g(t)\,dt\,. 
\qquad \forall g\in L^{p'}(0,\infty;Y)\,.
\end{equation}
The integrals in \eqref{e:s-star} and \eqref{e:t-star} are the sought 
respective representations of the adjoint operators in \eqref{e:duali}.

\end{proof}

\medskip

% THANKS

{\small
\section*{Acknowledgements}
\noindent
The research of F.~Bucci was partially supported by the Universit\`a degli Studi di Firenze
under the 2019 Project {\em Metodi ed Applicazioni per Equazioni Differenziali Ordinarie e a Derivate Parziali}.
Bucci is a member of the Gruppo Nazionale per l'Analisi Mate\-ma\-tica, la Probabilit\`a 
e le loro Applicazioni (GNAMPA) of the Istituto Nazionale di Alta Matematica (INdAM),
and participant to the GNAMPA Projects
{\em Controllabilit\`a di PDE in modelli fisici e in scienze della vita} (2019)
and {\em Problemi inversi e di controllo per equazioni di evoluzione e loro applicazioni} (2020).
She is also a member of the French-German-Italian Laboratoire International
Associ\'e (LIA) COPDESC in Applied Analysis.
\\
The research of P.~Acquistapace was partially supported by %the Universit\`a di Pisa, under 
the PRIN-MIUR Project 2017FKHBA8 of the Italian Education, University and Research Ministry.
}
\medskip

% REFERENCES

% FINE

\end{document}